\pdfoutput=1
\RequirePackage{scrlfile}
\AfterClass{article}{\RequirePackage{xr-hyper}}
\AfterPackage{hyperref}{
	\hypersetup{
		colorlinks=false,
		pdfborder={0 0 0},
		pdfborderstyle={/S/U/W 0},
	}
}

\documentclass[USenglish,cleveref,notab,nolineno]{socg-lipics-v2021}
\makeatletter
\title{Efficient two-parameter persistence computation via cohomology}
\author{Ulrich Bauer}{Department of Mathematics, TUM School of Computation, Information and Technology, and Munich Data Science Institute, Technical University of Munich (TUM), Germany \and
\url{www.ulrich-bauer.org}}{mail@ulrich-bauer.org}{https://orcid.org/0000-0002-9683-0724}{Supported by the German Research Foundation (DFG) through the Collaborative Research Center SFB/TRR 109 \emph{Discretization in Geometry and Dynamics} -- 195170736.}\author{Fabian Lenzen}{Department of Mathematics, TUM School of Computation, Information and Technology, Technical University of Munich (TUM), Germany}{fabian.lenzen@tum.de}{https://orcid.org/0000-0001-9579-6854}{}
\author{Michael Lesnick}{Department of Mathematics, SUNY Albany, USA}{mlesnick@albany.edu}{https://orcid.org/0000-0003-1924-3283}{}
\authorrunning{U. Bauer, F. Lenzen, M. Lesnick}
\Copyright{Ulrich Bauer, Fabian Lenzen, Michael Lesnick}

\EventEditors{Erin W. Chambers and Joachim Gudmundsson}
\EventNoEds{2}
\EventLongTitle{39th International Symposium on Computational Geometry (SoCG 2023)}
\EventShortTitle{SoCG 2023}
\EventAcronym{SoCG}
\EventYear{2023}
\EventDate{June 12--15, 2023}
\EventLocation{Dallas, Texas, USA}
\EventLogo{}
\SeriesVolume{258}
\ArticleNo{15}

\newcommand{\arxivID}{2303.11193}

\usepackage{mathtools, amssymb, tikz, tikz-cd, booktabs,apptools}

\usetikzlibrary{calc}

\newif\iffullversion
\fullversiontrue

\iffullversion
	\hideLIPIcs
\else
	\relatedversiondetails[linktext={arXiv:\arxivID}, cite=BauerLenzenEtAl:2023]{Full Version}{https://arxiv.org/pdf/\arxivID.pdf} 
	\usepackage{xr-hyper}

\fi

\newcommand{\CExtra}[3]{
	\AtAppendix{\crefalias{#1}{appendix-#1}}
	\iffullversion
		\crefname{appendix-#1}{#2}{#3}
	\else
		\crefformat{appendix-#1}{#2 ##2##1##3 in the full version}
		\crefrangeformat{appendix-#1}{#3~##3##1##4 to~##5##2##6 in the full version}
		\crefmultiformat{appendix-#1}{#3~##2##1##3}{ and~##2##1##3 in the full version}{, ##2##1##3}{and~##2##1##3 in the full version}
	\fi
}
\CExtra{algocf}{Algorithm}{Algorithms}
\CExtra{table}{Table}{Tables}
\CExtra{remark}{Remark}{Remarks}

\crefname{equation}{}{}
\Crefname{equation}{}{}

\usepackage{tikz}
\usetikzlibrary{calc}

\usetikzlibrary{shapes, shapes.misc, backgrounds}

\tikzset{%
	on layer/.code={
		\pgfonlayer{#1}\begingroup
		\aftergroup\endpgfonlayer
		\aftergroup\endgroup
}}

\newlength{\ModuleDiagramSymbolSize}
\tikzset{
	really densely dotted/.style={line cap=round, dash pattern=on 0pt off 2\pgflinewidth},
	module diagram symbol size/.style={
		generator/.style={shape=circle, fill, outer sep=0, inner sep=#1},
		relation/.style={shape=diamond, fill, outer sep=0, inner sep=#1},
		syzygy/.style={shape=cross out, draw, outer sep=0, inner sep=#1}
	},
	every module diagram/.style={
		module diagram symbol size=.2ex,
		line cap=round,
	},
	module diagram/.style={
		every module diagram,
		every picture
	}
}

\NewDocumentCommand{\CochainsAbsolute}{
	s
	O{} 
	m 
	O{} 
	m 
	O{below left} 
}{
	\begin{scope}[#2]
		\IfBooleanTF{#1}{
			\coordinate[at=(#3), label={[#2,#4]#6:#5}] (S);
			\coordinate[at=(#3 |- inf)] (R1);
			\coordinate[at=(#3 -| inf)] (R2);
			\coordinate[at=(inf)] (G);
		}{
			\node[syzygy, at=(#3), label={[#2,#4]#6:#5}] (S) {};
			\node[relation, at=(#3 |- inf)] (R1) {};
			\node[relation, at=(#3 -| inf)] (R2) {};
			\node[generator, at=(inf), opacity=.5] (G) {};
		}
		\fill[opacity=0.2] (S.center) |- (inf) |- cycle;
		\IfBooleanTF{#1}{
			\draw (inf -| #3) -- (S) -- (#3 -| inf);
		}{
			\draw
			(pinf -| #3) -- (S) -- (#3 -| pinf)
			(qinf -| #3) -- (R1) -- (G) -- (R2) -- (#3 -| qinf);
			\draw[every picture, really densely dotted]
			(pinf -| #3) -- (qinf -| #3)
			(pinf |- #3) -- (qinf |- #3);
		}
}

\NewDocumentCommand{\CochainsRelative}{%
	s
	O{} 
	m 
	O{} 
	m 
	O{above right} 
	O{} 
	O{} 
}{
	\begin{scope}[#2]
		\IfBooleanTF{#1}{
			\coordinate[at=(#3 |- inf), label={[#4]above:#7}] (G1);
			\coordinate[at=(#3 -| inf), label={[#4]right:#8}] (G2);
			\coordinate[at=(#3), label={[#4]#6:#5}] (R);
		}{
			\node[generator, at=(#3 |- inf), label={[#4]above:#7}] (G1) {};
			\node[generator, at=(#3 -| inf), label={[#4]right:#8}] (G2) {};
			\node[relation, at=(#3), label={[#4]#6:#5}] (R) {};
		}
		\fill[opacity=0.2] (#3) |- (P |- inf) |- (P -| inf) |- (#3);
		\IfBooleanTF{#1}{
			\draw (#3 |- inf) -- (R) -- (#3 -| inf);
		}{
			\draw
			(#3 |- qinf) -- (G1) -- (P |- inf)
			(#3 -| qinf) -- (G2) -- (P -| inf)
			(#3 |- pinf) -- (R) -- (#3 -| pinf);
			\draw[really densely dotted]
			(#3 |- pinf) -- (#3 |- qinf)
			(#3 -| pinf) -- (#3 -| qinf);
		}
	\end{scope}
}

\newcommand{\infinityCoordinate}{4,4}
\newcommand{\coordinates}{
	\coordinate (inf) at (\infinityCoordinate);
	\coordinate (pinf) at ($(inf)-(0.75,0.75)$);
	\coordinate (qinf) at ($(inf)-(0.25,0.25)$);
	\coordinate (O) at (0,0);
	\coordinate (P) at (-0.5,-0.5);
}

\ProvideDocumentCommand{\DrawGrid}{O{0} D(){0.5} m O{0} D(){0.5} m}{
	\begin{scope}[on background layer]
		\foreach \i in {#1, #2, ..., #3}{
			\foreach \j in {#4, #5, ..., #6}{
				\fill[gray] (\i, \j) circle (0.3pt);
			}
		}
	\end{scope}
}

\newcommand{\AxesHomology}{
	\coordinates
	\begin{scope}[]
		\draw[->] (-0.75,0) -- (O -| inf);
		\draw[->] (0,-0.75) -- (O |- inf);
	\end{scope}
}
\NewDocumentCommand{\AxesCohomology}{
	s 
}{
	\coordinates
	\begin{scope}[]
		\IfBooleanTF{#1}{
			\draw[<-, shorten <=-8] (0,0) -- (O -| inf);
			\draw[<-, shorten <=-8] (0,0) -- (O |- inf);
		}{
			\draw[<-, shorten <=-8] (0,0) -- (O -| pinf);
			\draw[<-, shorten <=-8] (0,0) -- (O |- pinf);
			\draw[really densely dotted]  (O -| pinf) -- (O -| qinf);
			\draw[really densely dotted] (O |- pinf) -- (O |- qinf);
			\draw          (O -| qinf) -- (O -| inf);
			\draw          (O |- qinf) -- (O |- inf);
		}
	\end{scope}
}

\newcommand{\FreeModule}[3][]{
	\begin{scope}[#1]
		\node[#3, at=(#2)] {};
		\coordinate[at=(#2)] (n) {};
		\fill[opacity=.25] (#2) rectangle (inf);
		\draw (n) -- (n |- inf) (n) -- (n -| inf);
	\end{scope}
}

\newcommand{\InjectiveModule}[2][]{
	\path (#2) coordinate (n);
	\fill[#1, opacity=.25] (n) rectangle (P);
	\draw[#1] (n) -- (n |- P) (n) -- (n -| P);
}

\newcommand{\ie}{i.e.}

\newcommand{\eg}{e.g.}

\newcommand{\wrt}{w.r.t.}

\renewcommand{\wrt}{with respect to}

\newcommand{\textiff}{iff}
\newcommand{\PH}{\textsc{ph}}
\newcommand{\PC}{\textsc{pc}}

\newcommand{\MFR}{\textsc{mfr}}

\let\xto\xrightarrow

\newcommand{\R}{\mathbf{R}}
\newcommand{\N}{\mathbf{N}}
\newcommand{\Z}{\mathbf{Z}}
\newcommand{\ZZ}{\mathbf{Z}^2}
\newcommand{\VR}{\mathit{VR}}
\newcommand{\Hom}{\mathrm{Hom}}
\newcommand{\IHom}{\mathsf{Hom}}

\newcommand{\vect}{\mathbf{vec}}

\newcommand{\rg}{\mathrm{rg}}
\newcommand{\cg}{\mathrm{cg}}
\renewcommand{\d}{\delta}
\newcommand{\id}{\mathrm{id}}
\DeclareMathOperator{\diam}{diam}
\DeclareMathOperator{\rk}{rk}
\DeclareMathOperator{\im}{im}
\DeclareMathOperator{\colim}{colim}

\DeclareMathOperator{\piv}{piv}
\DeclareMathOperator{\Lim}{Lim}
\DeclareMathOperator{\Colim}{Colim}

\DeclareMathOperator{\lpiv}{l-piv}
\DeclareMathOperator{\cpiv}{c-piv}
\mathchardef\mhyphen="2D
\newcommand{\Pers}[1]{#1\mhyphen\mathbf{pers}}

\newcommand{\Proj}[1]{\mathcal{P}_{\Pers{#1}}}
\newcommand{\Inj}[1]{\mathcal{I}_{\Pers{#1}}}
\newcommand{\lex}[1]{#1_{\mathrm{lex}}}
\newcommand{\colex}[1]{#1_{\mathrm{colex}}}
\newcommand{\Db}{\mathcal{D}^\mathrm{b}}

\DeclarePairedDelimiterX{\Set}[1]{\{}{\}}{\setargs{#1}}
\NewDocumentCommand{\setargs}{>{\SplitArgument{1}{;}}m}{\setargsaux#1}
\NewDocumentCommand{\setargsaux}{mm}{\IfNoValueTF{#2}{#1} {#1\nonscript\:\delimsize\vert\allowbreak\nonscript\:\mathopen{}#2}}%
\DeclarePairedDelimiterX{\Span}[1]{\langle}{\rangle}{\setargs{#1}}
\DeclarePairedDelimiter{\Shift}{\langle}{\rangle}
\DeclarePairedDelimiter{\abs}{\lvert}{\rvert}

\NewDocumentCommand{\Mtx}{s}{
	\IfBooleanTF{#1}{\@MtxStar}{\@Mtx}
}
\usepackage{mleftright}
\NewDocumentCommand{\@Mtx}{om}{
	\IfNoValueTF{#1}{
		\mleft(\begin{smallmatrix}#2\end{smallmatrix}\mright)
	}{
		\mathopen#1(\begin{smallmatrix}#2\end{smallmatrix} \mathclose#1)
	}
}
\NewDocumentCommand{\@MtxStar}{O{c}om}{
	\IfNoValueTF{#2}{
		\mleft(\begin{smallmatrix*}[#1]#3\end{smallmatrix*}\mright)
	}{
		\mathopen#2(\begin{smallmatrix*}[#1]#3\end{smallmatrix*} \mathclose#2)
	}
}

\newenvironment{smallcases}{\left\{\begin{smallmatrix*}[l]}{\end{smallmatrix*}\right.}

\DeclareRobustCommand\vdots{\mathpalette\@vdots{}}
\newcommand*{\@vdots}[2]{%
	\sbox0{$#1\cdotp\cdotp\cdotp\m@th$}%
	\sbox2{$#1.\m@th$}%
	\vbox{%
		\dimen@=\wd0 %
		\advance\dimen@ -3\ht2 %
		\kern.5\dimen@
		\dimen@=\wd2 %
		\advance\dimen@ -\ht2 %
		\dimen2=\wd0 %
		\advance\dimen2 -\dimen@
		\vbox to \dimen2{%
			\offinterlineskip
			\copy2 \vfill\copy2 \vfill\copy2 %
		}%
	}%
}
\DeclareRobustCommand\ddots{\mathinner{\mathpalette\@ddots{}\mkern\thinmuskip}}
\newcommand*{\@ddots}[2]{%
	\sbox0{$#1\cdotp\cdotp\cdotp\m@th$}%
	\sbox2{$#1.\m@th$}%
	\vbox{%
		\dimen@=\wd0 %
		\advance\dimen@ -3\ht2 %
		\kern.5\dimen@
		\dimen@=\wd2 %
		\advance\dimen@ -\ht2 %
		\dimen2=\wd0 %
		\advance\dimen2 -\dimen@
		\vbox to \dimen2{%
			\offinterlineskip
			\hbox{$#1\mathpunct{.}\m@th$}%
			\vfill
			\hbox{$#1\mathpunct{\kern\wd2}\mathpunct{.}\m@th$}%
			\vfill
			\hbox{$#1\mathpunct{\kern\wd2}\mathpunct{\kern\wd2}\mathpunct{.}\m@th$}%
		}%
	}%
}

\DeclareMathVersion{sansbold}
\SetSymbolFont{operators}{sansbold}{OT1}{lmss}{bx}{n}
\SetSymbolFont{letters}{sansbold}{OML}{cmssm}{b}{it} 
\SetSymbolFont{symbols}{sansbold}{OMS}{cmsssy}{b}{n} 
\SetSymbolFont{largesymbols}{sansbold}{OMX}{cmssex}{m}{n} 
\SetMathAlphabet\mathsf{sansbold}{OT1}{lmss}{bx}{n}
\SetMathAlphabet\mathbf{sansbold}{OT1}{lmss}{bx}{n}
\SetSymbolFont{AMSa}{sansbold}{U}{ssmsa}{m}{n}
\SetSymbolFont{AMSb}{sansbold}{U}{ssmsb}{m}{n}
\SetMathAlphabet{\mathfrak}{sansbold}{U}{euf}{b}{n}
\g@addto@macro\bfseries{\mathversion{sansbold}}

\usepackage[noend, nosemicolon]{algorithm2e}
\renewcommand{\@algocf@capt@plain}{above}
\SetKwProg{Fn}{function}{:}{end}
\newcommand{\FunctionName}[1]{\textnormal{\texttt{#1}}}
\newcommand\NewFunction[2]{\newcommand{#1}[1]{\FunctionName{#2}\ensuremath{(##1)}}}
\NewFunction{\Ker}{\hyperref[alg:lw-algorithm]{Ker}}
\NewFunction{\MgsWithKer}{\hyperref[alg:ker-mgs]{MGSWithKer}}
\NewFunction{\Factorize}{\hyperref[algo:factorize]{Factorize}}
\NewFunction{\Minimize}{\hyperref[algo:minimize]{Minimize}}
\NewFunction{\MinimizeChain}{\hyperref[algo:chain-chunk]{MinimizeChainComplex}}
\NewFunction{\MinimizeCochain}{\hyperref[algo:cochain-chunk]{MinimizeCohainComplex}}
\NewFunction{\PopHeap}{PopMin}
\NewFunction{\PushHeap}{Push}
\NewFunction{\Bireduce}{\hyperref[alg:bigraded reduction]{Bireduce}}
\NewFunction{\Sparsify}{Sparsify}
\NewFunction{\Swap}{Swap}
\SetKwFor{Forever}{forever}{do}{end}
\SetKw{Yield}{yield}
\SetKw{Break}{break}
\SetKw{Continue}{continue}
\SetKw{Return}{return}
\SetKwComment{tcp}{$\triangleright$\ }{}
\SetCommentSty{itshape}

\DeclareFontShape{T1}{lmr}{m}{scit}{<->ssub*lmr/m/scsl}{}
\begin{document}
\maketitle
\begin{abstract}%
	Clearing is a simple but effective optimization for the standard algorithm of persistent homology (\PH),
	which dramatically improves the speed and scalability of  \PH\ computations for Vietoris--Rips filtrations.
	Due to the quick growth of the boundary matrices of a Vietoris--Rips filtration with increasing dimension,
	clearing is only effective when used in conjunction with a dual (cohomological) variant of the standard algorithm.
	This approach has not previously been applied successfully to the computation of two-parameter \PH.

	We introduce a cohomological algorithm for computing minimal free resolutions of two-parameter \PH\ that allows for clearing.
	To derive our algorithm, we extend the duality principles which underlie the one-parameter approach to the two-parameter setting.
	We provide an implementation and report experimental run times for function-Rips filtrations.
	Our method is faster than the current state-of-the-art by a factor of up to 20.
\end{abstract}

\section{Introduction}
\subparagraph*{Motivation}
Persistent homology \cite{EdelsbrunnerLetscherEtAl:2002, ZomorodianCarlsson:2005, OtterPorterEtAl:2017} analyzes how the homology of a filtered topological space changes as the filtration parameter increases.
By assigning filtered spaces (\eg, Vietoris--Rips filtrations) to data sets, it provides simple signatures of the data called \emph{barcodes}, which encode multi-scale information about the shape of the data.
Thanks to recent advances in \PH\ computation and software \cite{Bauer:2019, ZhangXiaoEtAl:2020, AggarwalPeriwal:2021, PerezHaukeEtAl:2021, gudhi:urm, HenselmanGhrist:2017, Henselman-Petrusek:2021},
\PH\ has become popular for practical data applications \cite{giunti2021TDA}. 
A well-known stability result \cite{Cohen-SteinerEdelsbrunnerEtAl:2007, BauerLesnick:2015, BlumbergLesnick:2022a} guarantees that small perturbations (in the Gromov--Hausdorff distance)  of the input data lead to small perturbations (in the bottleneck distance) of the barcodes of the \PH\ of Vietoris–Rips filtrations.
However, Vietoris--Rips \PH\ is notoriously unstable to outliers.
Besides other strategies \cite[Section 1.7]{BlumbergLesnick:2022a}, a commonly proposed remedy for this is the introduction of a second filtration parameter controlling the local density of the point cloud, which leads to the notions of \emph{two-} and \emph{multi-parameter} \PH\ \cite{CarlssonZomorodian:2009,Sheehy:2012,CorbetKerberEtAl:2021a, LesnickWright:2015,rolle2020multiparameter,BotnanLesnick:2022}.
A central computational problem of two-parameter \PH, somewhat akin to the computation of a barcode, is the computation of a \emph{minimal free presentation} or a \emph{minimal free resolution} (\MFR) of the \PH\ module.
While a resolution contains more information than a presentation, the underlying algorithmic problems are essentially the same, and we focus on the computation of a \MFR\ in this work.
Such a resolution is often quite small in practice \cite{FugacciKerberEtAl:2023}, and computing it is a natural first step in computing invariants or metrics in the multi-parameter setting \cite{BotnanLesnick:2022}.

Computing a \MFR\ of two-parameter \PH\ is more involved than in the one parameter case.
The problem can be solved by classical Gr\"obner basis algorithms, which work in much greater generality but do not scale well enough for practical TDA applications \cite{LesnickWright:2022}.
Recently, a specialized algorithm was introduced \cite{LesnickWright:2022,KerberRolle:2021}, which
is far more efficient than the Gröbner basis approach, both in theory and in practice.  This has substantially lowered the barrier to practical data analysis with two-parameter persistence \cite{kerber2020multi}.
For recent applications, see, e.g., \cite{benjamin2022multiscale,vipond2021multiparameter}.

Nevertheless, existing software for two-parameter \PH\ is much slower and less scalable than one-parameter \PH\ implementations such as \cite{Bauer:2019, AggarwalPeriwal:2021, PerezHaukeEtAl:2021,ZhangXiaoEtAl:2020}.
The approach of  \cite{LesnickWright:2022,KerberRolle:2021} boils down to a matrix reduction scheme similar to the standard algorithm of one-parameter \PH~\cite{ZomorodianCarlsson:2005},
and has the same asymptotic run time, cubic in the size of the complex.
However, modern one-parameter \PH\ algorithms incorporate several critical optimizations.
In particular, it is known that for Vietoris--Rips filtrations, \emph{clearing} \cite{BauerKerberEtAl:2014b, ChenKerber:2011} leads to major performance gains when combined with a dual (cohomological) variant of the standard persistence algorithm \cite{deSilvaMorozovEtAl:2011,BauerKerberEtAl:2014b}.
All state-of-the-art software for computing Vietoris--Rips \PH\ employ this strategy.
In two parameters, however, working with persistent cohomology (\PC) is more challenging, essentially because, in contrast to one parameter, relative simplicial cochains of filtered complexes do not form free modules.

\subparagraph*{Contributions}
In order to compute \MFR s of \PH\ of function-Rips bifiltrations more efficiently,
we introduce a cohomological variant of the algorithm of \cite{LesnickWright:2022,KerberRolle:2021},
which we outline now.

Let $K_*$ be a finite simplicial $\Z^n$-filtration with $K = \bigcup_{z \in \Z^n} K_z$.  Assume that $K_*$ is one-critical, i.e., the set $\Set{z ; \sigma \in K_z}$ has a unique minimal element $g(\sigma)$ for every $\sigma \in K$.  We define a certain cochain complex $N^\bullet(K_*)$ of free persistence modules.  In this paper, $H_\bullet(K)$ always denotes the reduced simplicial homology of $K$.
If $H_d(K_z) = 0$ for all $d$ but finitely many indices $z\in \Z^n$ (which can easily be ensured by adding additional simplices to $K_*$),
then $H^{d + n}(N^\bullet(K_*))$ is isomorphic to the dual module of $H_d(K_*)$ for all $d$ (see \cref{thm:UQ-LQ-isomorphism}).
This can be seen as a generalization of a corresponding statement for one-parameter persistence, in which case $N^\bullet(K_*)$ equals the relative cochain complex $C^\bullet(K, K_*)$ \cite[Theorem~2.4]{deSilvaMorozovEtAl:2011}; see also \cite{BauerSchmahl:2021a}.

Given a (minimal) free resolution $F_\bullet$ of an $n$-parameter persistence module $M$ and a choice of basis for each module of $F_\bullet$, we show that the matrices representing $F_\bullet$ also represent a (minimal) injective resolution of $M$; see \cref{thm:calabi-yau}.
In particular, this allows us to easily convert a (minimal) free resolution of a module (\eg, $H^{d + n}(N^\bullet(K_*))$) to a (minimal) free resolution of its dual (\ie, $H_d(K_*)$); see \cref{thm:calabi-yau for matrices}.

For $n = 2$, we propose a method to compute a \MFR\ of $H^{d + 2}(N^\bullet(K_*))$ (and thus $H_d(K_*))$ solely from the coboundary map $\d^{d+1}\colon N^d(K_*) \to N^{d+1}(K_*)$; see \cref{sec:computing cohomology}.
At the core of this method is an algorithm for the following problem:  given a morphism $f\colon F \to F'$ of free persistence modules and a basis of the vector space $\colim \im f$, compute a basis of the free persistence module $\ker f$;
see \cref{thm:correctness of bireduction algorithm}.
The algorithm is compatible with the clearing optimization, which improves its performance considerably.

We have implemented our approach \cite{Lenzen:2023} and report timing results from computational experiments with function-Rips bifiltrations.
On most instances considered, our approach is significantly faster than the approach \cite{FugacciKerberEtAl:2023} used in \texttt{mpfree}, and on certain instances, our implementation is able to outperform the approach of by a factor of up to 20.

A number of recent methods for computation of multiparameter persistence focus on decreasing the size of the input complex without changing its homology \cite{FugacciKerber:2019a, ScaramucciaIuricichEtAl:2020, FugacciKerberEtAl:2023,AlonsoKerberEtAl:2022}.  These methods can be used as a preprocessing step to the computation of a minimal free resolution.
In our computational experiments, we explore the effect of the chunk preprocessing method of \cite{FugacciKerberEtAl:2023} on the efficiency of our method.
We find that in our experiments, our method generally performs better without this preprocessing.
In contrast, we observe that the preprocessing is very helpful for the the approach of \cite{LesnickWright:2022,KerberRolle:2021},
as previously reported \cite{FugacciKerberEtAl:2023}.
We also observe that applying the chunk algorithm on cochain complexes instead of chain complexes may significantly increase the performance even for homology computation.

\section{Background}
\subsection{Persistence modules}\label{Sec:Pers_Mod}
Let $k$ be a field, let $n \in \N$, let $\vect$ denote the category of finite dimensional $k$-vector spaces, and consider $\Z^n$ as a poset with the usual product partial order.
A (pointwise finite dimensional) \emph{$\Z^n$-persistence module}, also called an \emph{$n$-parameter persistence module}, is a functor $M\colon \Z^n\to \vect$.
The maps $M_{z \leq z'}\colon M_z \to M_{z'}$ are called the \emph{structure maps} of $M$.
If $m\in M_z$, we call $z$ the \emph{grade} of $m$, denoted by $g(m)$.
The \emph{total dimension} of $M$ is $\sum_{z\in \Z^n} \dim M_z$.
We write $\Pers{\Z^n}$ for the abelian category of pointwise finite dimensional $\Z^n$-persistence modules.
Its morphisms are natural transformations.  $\Pers{\Z^n}$ is equivalent to a full subcategory of the category of multigraded modules over the ring $k[x_1, \dotsc, x_n]$; see \cite{CarlssonZomorodian:2009}.
The algebra $k[x_1, \dotsc, x_n]$ is not a principal ideal domain unless $n = 1$; therefore, $\Z^n$-persistence modules cannot be described by a barcode for $n > 1$.

Let $V^* = \Hom_k(V, k)$ denote the dual of a $k$-vector space $V$.
The \emph{dual} of a $\Z^n$-persistence module $M$ is the $\Z^n$-persistence module $M^*$ with $(M^*)_z = (M_{-z})^*$ and $(M^*)_{z \leq z'} = (M_{-z'\leq -z})^*$.
An object $M$ of an abelian category $\mathcal{C}$ is \emph{projective} (respectively \emph{injective}) if the functor $\Hom_{\mathcal{C}}(M, -)$ (respectively $\Hom_{\mathcal{C}}(-, M)$) is exact.
The duality $M \mapsto M^*$ is an exact contravariant equivalence of categories and thus maps projective to injective modules and vice versa.

For $z \in \Z^n$, let $F(z)$ be the module with $F(z)_{w} = \begin{smallcases} k & \text{if $z \leq w$},\\ 0 & \text{otherwise,} \end{smallcases}$ and $F(z)_{w \leq w'} = \id_k$ if $z \leq w \leq w'$.
A module $F$ is \emph{free} if there are elements $(z_i)_{i \in I} \subseteq \Z^n$, for some indexing set $I$, such that there is an isomorphism $b\colon \bigoplus_{i\in I} F(z_i) \to F$.
Every finitely generated projective persistence module is free \cite{Quillen:1976,Suslin:1976,HoppnerLenzing:1981}.
Let $e_i$ denote the element $1 \in F(z_i)_{z_i}$ of the component of $\bigoplus_{i \in I} F(z_i)$ indexed by $i$.
The set $\Set{b(e_i) ; i \in I}$ is a \emph{basis} of $F$.  The multiset $\rk F \coloneqq \Set{z_i ; i \in I}$ is uniquely determined by $F$ and called its \emph{(graded) rank}.
A module $M$ is \emph{finitely generated} if there is a pointwise surjection $F \to M$ from a free module $F$ of finite rank; the image of a basis of $F$ under such a map is called a \emph{generating system} of $M$.

A \emph{graded matrix} $M$ is a matrix with entries $M_{ij} \in k$ whose rows and columns are decorated with row grades $\rg^M_*$ and column grades $\cg^M_*$.
The \emph{graded transpose} of a graded $m\times n$-matrix $M$ is the graded $n\times m$-matrix $M^T$ with entries $(M^T)_{ij} = M_{m+1-j,n+1-i}$, row grades $\rg^{M^T}_i = -\cg^M_{n+1-i}$ and column grades $\cg^{M^T}_j = -\rg^M_{m+1-i}$.
A morphism $f\colon F \to F'$ of finite rank free modules $F$ and $F'$ with respective bases $b_1, \dotsc, b_n$ and $b'_1, \dotsc, b'_m$
is uniquely represented by a graded $m \times n$-matrix $M$ with $\cg^M_j = g(b_j)$, $\rg^M_i = g(b'_i)$, and entries $M_{ij}$ such that $f(b_j) = \sum_i M_{ij} F'_{g(b'_i) \leq g(b_j)} (b'_i)$ for all $j$.

\begin{lemma}
	\label{thm:graded matrices free modules}
	A graded matrix $M$  represents a morphism of finite rank free modules \textiff\ $M_{ij} =0$ whenever $\rg^M_i \not \leq \cg^M_j$.
\end{lemma}

\begin{proof}
This follows from the fact that $\Hom(F(z), F(z')) = \begin{smallcases}
	k & \text{if $z \geq z',$} \\
	0 & \text{otherwise.}
\end{smallcases}$
\end{proof}

If bases of free modules $F, F'$ are fixed, we identify a morphism $F \to F'$ with the graded matrix representing it.
A \emph{free resolution} (resp., \emph{injective} resolution) of a module $M$ is a chain complex $F_\bullet\colon \dotsb \to F_1 \to F_0$ of free modules (resp., cochain complex $I^\bullet\colon I^0 \to I^1 \to \dotsb$ of injective modules) concentrated in non-negative degrees that is quasi-isomorphic to $M$.
A \emph{(homological) $d$-ball} is an acyclic chain complex of the form $\dotsb \to 0 \to F(z) \stackrel\id\to F(z) \to 0 \to \dotsb$ for some $z\in \Z^n$, concentrated in degrees $d, d-1$.
A free resolution, and, more generally, a chain complex of free modules, is called \emph{minimal} if it contains no direct summand isomorphic to a ball.
An injective resolution is \emph{minimal} if its dual is minimal.
A morphism $F_1\to F_0$ of free modules is called a \emph{(minimal) free presentation} of a module $M$ if it extends to a (minimal) free resolution of $M$.

\begin{theorem}[see {\cite[Theorem~20.2]{Eisenbud:1995}, \cite[Theorem~7.5]{Peeva:2011}}]
	\label{thm:existence of minimal free resolutions}
	Every finitely generated module has a \MFR.
	Every free resolution is isomorphic to the direct sum of a \MFR\ with a direct sum of homological balls.
	In particular, a \MFR\ is unique up to isomorphism of chain complexes.
\end{theorem}
Thus, letting $F_\bullet$ be a \MFR\ of a finitely generated persistence module $M$, the graded ranks $\beta_{q}(M) \coloneqq \rk F_q$, called the \emph{graded Betti numbers} of $M$, are independent of the choice of $F_\bullet$.
\begin{theorem}[{Hilbert's Syzygy theorem \cite[Theorem~15.2]{Peeva:2011}, \cite[Corollary~19.7]{Eisenbud:1995}}]
	\label{thm:syzygy}
	Every $\Z^n$-persistence module has a \MFR\ of length at most $n$.
\end{theorem}

\subsection{Filtrations}\label{Sec:Filtered}
For $P$ any poset, a (simplicial) $P$-\emph{filtration} is a functor $K_*$ from $\Z^n$ to the category of simplicial complexes such that the simplicial maps $K_{z \leq z'}$ are inclusions.
We write $K = \bigcup_{z \in \Z^n} K_z$.
\begin{example}
	Let $S$ be a metric space and let $\diam \sigma = \max_{s, t \in \sigma} d(s, t)$ for every finite, non-empty $\sigma \subseteq S$.
	The \emph{Vietoris--Rips filtration} $\widehat{\VR}_*(S)$ associated to $S$ is the $\R$-filtration given by $\widehat{\VR}_r(S) = \Set{\sigma \subseteq S; 0 < \abs{\sigma} < \infty, \diam \sigma \leq r}$.
	If $S$ is finite and non-empty, let $r_1 < r_2 < \dotsb < r_n$ be the distinct values $\diam \sigma$ can attain for $\sigma \subseteq S$.
	By setting \[
		\VR_i(S) = \begin{smallcases}
			\widehat{\VR}(S)_{r_1} & \text{if $i \leq 1$,}\\
			\widehat{\VR}(S)_{r_i} & \text{if $1 <i < n$,}\\
			\widehat{\VR}(S)_{r_n} & \text{if $n \leq i$,}\\
		\end{smallcases}
	\]
	we obtain a $\Z$-filtration $\VR_*(S)$, which we also call a Vietoris--Rips filtration.
\end{example}

If $K_*$ is a $\Z^n$-filtration, its absolute and relative simplicial chains $C_\bullet(K_*)$ and $C_\bullet(K, K_*)$ (with coefficients in $k$) form chain complexes of $\Z^n$-persistence modules, and the respective cycles $Z_d(-)$, boundaries $B_d(-)$ and homology $H_d(-)$ are $\Z^n$-persistence modules for all $d$.
The dual cochain complex $C^\bullet \coloneqq (C_\bullet)^*$ of a chain complex $C_\bullet$ has components $C^d = (C_d)^*$.
Its cocycles $Z^d(C^\bullet)$, coboundaries $B^d(C^\bullet)$ and cohomology $H^d(C^\bullet)$ are $\Z^n$-persistence modules for all $d$.
Because $(-)^*$ is exact, there is a natural isomorphism $H^d(C^\bullet) \to H_d(C_\bullet)^*$ for all $d$.
Our indexing convention is that a chain complex $C_\bullet$ has the boundary morphisms $\partial_d\colon C_d \to C_{d-1}$ and a cochain complex $C^\bullet$ has coboundary morphisms $\d^d = (\partial_d)^*\colon C^{d-1} \to C^d$.
This differs from the standard convention, but is chosen such that $C^\bullet = (C_\bullet)^*$ has $\d^d = (\partial_d)^*$.

A $\Z^n$-filtration $K_*$ is \emph{one-critical} if for every $\sigma \in K$ the set $\Set{z \in \Z^n; \sigma \in K_z}$ has a unique minimal element $g(\sigma)$, called the \emph{grade} of $\sigma$.
In this case, $C_\bullet(K_*) = \bigoplus_{\sigma \in K_*} F(g(\sigma))$ is free,
with a basis $\Set{e_\sigma; \sigma \in K_*}$ satisfying $g(e_\sigma) = g(\sigma)$.
In particular, $\partial_\bullet$ can be represented by a graded matrix $[\partial_\bullet]$.

\begin{example}
	\label{def:function-Rips}
	The \emph{function-Rips bifiltration} $\VR^f_*(S)$ associated to a finite metric space $S$ and a function $f\colon S \to \Z$ is the one-critical $\ZZ$-filtration with $\VR^f_{x,y} = \Set{ \sigma \in \VR_y(S); \max_{s \in \sigma} f(s) \leq x}$.
 \cref{fig:sample S1} illustrates the Hilbert function and graded Betti-numbers of the \PH\ of a function-Rips bifiltration, where the function is a density function.
\end{example}

\begin{figure}
	\centering
		\includegraphics[scale=0.75]{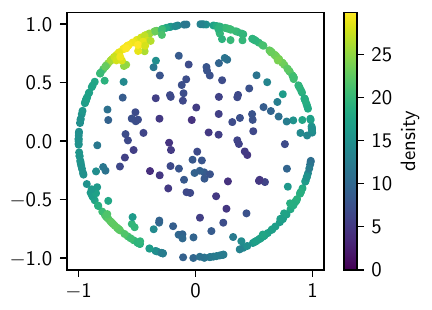}\hfill
		\includegraphics[scale=0.75]{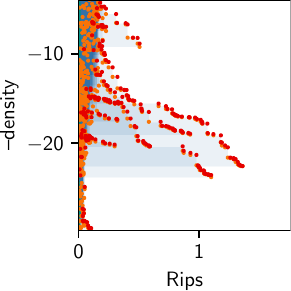}\hfill
		\includegraphics[scale=0.75]{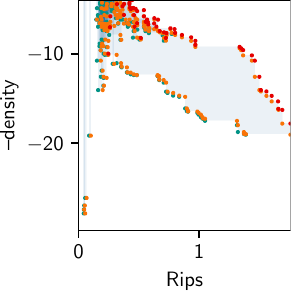}
	\caption[]{%
		A point set $S$ (left) with $\abs{S}=400$, 
		with density function $\rho(p) \coloneqq \sum_{q \in S \setminus \{p\}} \exp\bigl(-\frac{d(p, q)^2}{2\sigma^2}\bigr)$ for $\sigma = 0.15$;
		graded Betti numbers (teal: $\beta_0$, red: $\beta_1$, orange: $\beta_2$) and Hilbert function (shades of blue increasing from dim = $0$ to dim $\geq 10$)
		of $H_0(\VR^\rho_*(S))$ (middle) and $H_1(\VR^\rho_*(S))$ (right).
	}
	\label{fig:sample S1}
\end{figure}
\subsection{One-parameter persistence and clearing}
We next turn attention to persistence modules over $\Z$.  For $-\infty \leq b_i < d_i \leq \infty$, let $I(b, d)$ be the \emph{interval module} with $I(b, d)_z = \begin{smallcases}k & \text{if $b \leq z < d$,} \\ 0 & \text{otherwise}\end{smallcases}$ and $I(b, d)_{z \leq z'} = \id_k$ if $b \leq z \leq z' < d$.
Every pointwise finite-dimensional $\Z$-persistence module is isomorphic to an essentially unique direct sum $\bigoplus_{i \in I} I(b_i, d_i)$ \cite{Webb:1985,ZomorodianCarlsson:2005}.
The collection of the pairs $(b_i,d_i)$ is called the \emph{barcode} of $M$.

Given a finite $\Z$-filtered complex $K_*$, one is usually interested in computing the barcode of $H_\bullet(K_*)$.
Since $k[x]$ is a principal ideal domain, the submodules $Z_d(K_*), B_d(K_*) \subseteq C_d(K_*)$ are free for all $d$.
The \emph{standard algorithm} \cite[\S 3]{EdelsbrunnerHarer:2008} computes bases of $Z_d(K_*)$ and $B_d(K_*)$ and thus the barcode of $H_\bullet(K_*)$ by applying an order-respecting Gaussian column reduction scheme to each graded matrix $[\partial_d]$.
Each relative cochain module $C^d(K, K_*)$ is also free, so the same algorithm computes the barcode of $H^\bullet(K, K_*)$ from the graded matrices $[\d^d] = [\partial_d]^T$ representing the coboundary operators $\d^d$.
The barcodes of $H_\bullet(K_*)$ and $H^{\bullet}(K, K_*)$ determine each other in a simple way, as is seen by considering the long exact sequence of the pair $(K, K_*)$ \cite{deSilvaMorozovEtAl:2011, BauerSchmahl:2021a}.

It has been observed that for Vietoris--Rips filtrations, computing $H^{\bullet}(K, K_*)$ instead of $H_\bullet(K_*)$ is far more efficient.
This increase in efficiency hinges on the use of the \emph{clearing} optimization scheme \cite{BauerKerberEtAl:2014b, ChenKerber:2011, Bauer:2019}, which we now explain.
The \emph{pivot} of a matrix column is the largest row index of a non-zero entry in that column.
The standard algorithm applies left-to-right column additions to bring $[\d^{d+1}]$ into \emph{reduced}
form $R^{d+1}$, meaning that all columns of $R^{d+1}$ have pairwise distinct pivots.
If a column $R^d_j$ is non-zero with pivot $i$, then $R^{d+1}_i = 0$.
Therefore, if $R^d$ is known from previous computations, the reduction of $[\d^{d+1}]_j$ to zero can be skipped.
As the standard algorithm would typically spend most of its run time on the columns of $[\d^{d+1}]$ that are reduced to zero,
skipping most of these accelerates the algorithm considerably.

If the reduced homology $H_d(K)$ is zero for all $d$, then the long exact sequence of the pair $(K, K_*)$ shows that $H^{d+1}(K, K_*)^* \cong H_d(K_*)$ for all $d$, so one would expect that they can be computed from the same data.
Indeed, one can use clearing to compute $H^{d+1}(K, K_*)$ (respectively $H_d(K_*)$) from $\d^{d+1}$ (respectively $\partial_{d+1}$) alone;
see \cref{algo:1D persistence clearing}.


\subsection{Computation of 2-parameter persistence}

\subparagraph*{The LW-Algorithm}
\label{sec:lw-algorithm}
Assume that $C_\bullet$ is a chain complex of free $\ZZ$-persistence modules of finite rank;
\eg, $C_\bullet=C_\bullet(K_*)$ for a one-critical $\ZZ$-filtration $K_*$,
and let $D_d$ be the matrix representing $\partial_d\colon C_d \to C_{d-1}$ for all $d$.
\Cref{thm:existence of minimal free resolutions,thm:syzygy} imply that the kernel of a morphism of finitely generated free $\Z^2$-modules is free.
In particular, $Z_d(C_\bullet)$ is free for all $d$, so the sequence
\(
	0 \to Z_{d+1}(C_\bullet) \xto{i_{d+1}} C_{d+1} \xto{p_{d+1}} Z_d(C_\bullet)
\)
is a free resolution of $H_d(K_*)$.
From $D_d$, the \emph{LW-Algorithm} \cite{LesnickWright:2022,KerberRolle:2021} (see \cref{alg:lw-algorithm}) computes a graded matrix $I_d$ representing $i_d\colon Z_d(C_\bullet) \hookrightarrow C_d$.
A variant of that algorithm (see \cref{alg:ker-mgs}) computes from $D_{d+1}$ a graded matrix $D'_{d+1}\colon C'_{d+1} \to C_d$,
whose columns represent a minimal generating system of $B_d(C_\bullet)$,
together with a graded matrix $I'_{d+1}$ representing the kernel $Z'_{d+1}$ of the morphism represented by $D'_{d+1}$.
There is a unique graded matrix $P'_{d+1}$ such that $D'_{d+1} = I_d P'_{d+1}$, which can be obtained by \cref{algo:factorize}.
Then $I'_{d+1}$ and $P'_{d+1}$ represent a free resolution
\begin{equation}
	\label{eq:homology free resolution}
	0 \to Z'_{d+1} \xto{I'_{d+1}} C'_{d+1} \xto{P'_{d+1}} Z_d(C_\bullet)
\end{equation}
 of $H_d(C_\bullet)$.
To obtain a \MFR, it remains to split off summands from \eqref{eq:homology free resolution} that are isomorphic to homological balls.
There is an embarrassingly parallel algorithm that computes a minimal chain complex quasi-isomorphic to a given one; see \cref{rmk:minimization}.
In particular, this algorithm can be used to convert a free resolution to a minimal one.
It can also be used to split off balls from the input complex $C_\bullet$.
This is known as \emph{chunk preprocessing} and typically improves performance of the LW-algorithm by a considerable amount
\cite{FugacciKerber:2019a,KerberRolle:2021,FugacciKerberEtAl:2023}.

\section{Cohomology computation}
Let $K_*$ be a one-critical $\Z^n$-filtration.
If $n > 1$, then neither $C^\bullet(K_*)$ nor $C^\bullet(K, K_*)$ are complexes of free modules.
Since the LW-algorithm assumes that the input complex is a complex of free modules, the strategy from \cref{sec:lw-algorithm} cannot be used to compute $H^\bullet(K_*)$ or $H^\bullet(K, K_*)$ directly.
Instead, we consider  a cochain complex $N^\bullet(K_*)$ that can be used to compute $H_d(K_*)$.

\subsection{The free cochain complex \texorpdfstring{$N^\bullet(K_*)$}{C*(K⁎)}}
For a module $M$ and $z \in \Z^n$, let $M\Shift{z}$ be the module with graded components $M\Shift{z}_w = M_{z+w}$.
For $z \geq 0$, the structure maps of $M$ give a morphism $M \to M\Shift{z}$.
Note that $M\Shift{z}^* = M^*\Shift{-z}$.
For a graded matrix $A$, let $A\Shift{z}$ be the graded matrix with $A \Shift{z}_{ij} = A_{ij}$, $\rg^{A\Shift{z}}_i = \rg^A_i + z$ and $\cg^{A\Shift{z}}_j = \rg^A_j + z$ for all $i,j$.

Fix a total order on the simplices of $K_*$, so that
the boundary map $\partial_\bullet$ of the chain complex $C_\bullet(K_*) = \bigoplus_{\sigma \in K_*} F(g(\sigma))$ is represented by the graded matrix $[\partial_\bullet]$.
Let $\epsilon = (1,\dotsc,1) \in \Z^n$.
Let \[N^\bullet(K_*) = \bigoplus_{\sigma} F(-g(\sigma)+\epsilon)\] be the cochain complex whose
coboundary operator $\d_{N}^\bullet$ is represented by $[\partial_\bullet]^T\Shift{-\epsilon}$ \wrt\ the standard basis.
It follows from \cref{thm:graded matrices free modules} that this is a well-defined cochain complex.
The key property of this chain complex is summarized in the following proposition, whose proof is deferred to the next subsection.

\begin{proposition}
	\label{thm:UQ-LQ-isomorphism}
	If $H_d(K_*)$ has finite total dimension for all $d$,
	then there is a natural isomorphism $H_d(K_*) \cong H^{d+n}(N^\bullet(K_*))^*$ for all $d$.
\end{proposition}

\begin{corollary}
	\label{thm:free resolution dualizes to injective resolution}
	If $H_\bullet(K_*)$ has finite total dimension and
	$F_\bullet$ is a free resolution of $H^{d+n}(N^\bullet(K_*))$,
	then $(F_\bullet)^*$ is an injective resolution of $H_d(K_*)$.
\end{corollary}

\subsection{The Calabi--Yau-property of persistence modules}
Besides proving \cref{thm:UQ-LQ-isomorphism}, we will need to convert the injective resolution of $H_d(K_*)$ from \cref{thm:free resolution dualizes to injective resolution} into a free resolution of $H_d(K_*)$.
Both will follow from \cref{thm:calabi-yau}, which establishes a property of persistence modules known as the \emph{Calabi--Yau property} in some areas of algebra \cite{Ginzburg:2007}; see \cite[Lemma~4.1]{Keller:2008} for a proof in a more general context.
As it turns out, there is a close correspondence between injective and free resolutions that we explore in this section.
For $z \in \Z^n$, we define the injective module $I(z) = F(-z)^*$; \ie,
$I(z)_w = \begin{smallcases} k & \text{if $w \leq z$,} \\ 0 & \text{otherwise,} \end{smallcases}$ and $I(z)_{w \leq w'} = \id$ if $w \leq w' \leq z$.

\begin{definition}
	For persistence modules $M, N$, let $\IHom(M, N)$ be the persistence module with components $\IHom(M, N)_z = \Hom(M, N\Shift{z})$.
\end{definition}

Let $\Proj{\Z^n}$ and $\Inj{\Z^n}$ be the full subcategories of $\Pers{\Z^n}$ consisting of free and injective modules, respectively.
\begin{lemma}[see {\cite[Proposition~2.10 in Chapter III]{AssemSimsonEtAl:2006}}]
	\label{thm:nakayama functor}
	The \emph{Nakayama functor}
	\[\nu \coloneqq \IHom(-, F(0))^*\colon \Proj{\Z^n} \to \Inj{\Z^n}\]
	is an equivalence of categories with quasi-inverse $\nu^{-1} = \IHom(I(0)^*, -)$.
\end{lemma}

One checks that $\nu F(z) = I(z)$ and $\nu^{-1} I(z) = F(z)$.
Therefore, $N^\bullet(K_*) = (\nu C_\bullet(K_*) \Shift\epsilon)^*$.
For a chain complex $C_\bullet$ and $i \in \Z$, let $C_\bullet[i]$ be the chain complex whose $d$th module is $(C_\bullet[i])_d = C_{i+d}$.
Analogously, for a cochain complex $C^\bullet$, let $C^\bullet[i]$ be the cochain complex with $(C^\bullet[i])^d = C^{i+d}$.
Note that $C_\bullet[i]^* = (C_\bullet)^*[i]$.

\begin{theorem}
	\label{thm:calabi-yau}
	If $F_\bullet$ is a complex of free $\Z^n$-persistence modules such that $H_d(F_\bullet)$ has finite total dimension for all $d$, then $F_\bullet$ and $\nu F_\bullet[n]\Shift{\epsilon}$ are naturally quasi-isomorphic.
\end{theorem}

\begin{proof}
	\label{proof:calabi-yau}
	For $z \in \Z^n$, we write $z = (z_1, \dotsc, z_n)$.
	For $n \in \N$, let $[n] \coloneqq \{ 1, \dotsc, n \}$
	and let $\binom{[n]}{k} \coloneqq \Set{ S \subseteq [n] ; \abs{S} = k }$.
	For $S = \{s_1 < \dotsc < s_k\} \in \binom{[n]}{k}$ and $z \in \Z^n$, $w \in \Z^k$, we let
	\[
	z|^S_w \coloneqq (z_1, \dotsc, z_{s_1-1}, w_1, z_{s_1+1}, \dotsc, z_{s_k-1}, w_k, z_{s_k+1},\dotsc,z_n)
	\]
	be the $n$-tuple obtained from $z$ by replacing the components indexed by $S$ by the entries of~$w$.
	For any module $M \in \Pers{\Z^n}$ and $S \subseteq [n]$, we let $\Colim_S M$ be the module with
	\begin{align*}
		(\Colim_S M)_z &= \colim_{w \in \Z^k} M_{z|^S_w},&
		(\Colim_S M)_{z \leq z'} &= \colim_{w \in \Z^k} M_{z|^S_w \leq z'|^S_w}.
	\end{align*}
	For example, for $n = 3$ we get
	\(
	(\Colim_{\{1, 3\}} M)_{(z_1, z_2, z_3)} = \colim_{(w_1, w_2) \in \Z^2} M_{(w_1, z_2, w_2)}.
	\)
	The module $\Colim_S M$ is constant along the axes specified by $S$.
	In particular, $\Colim M = \Colim_{[n]} M$ is the module that is constantly $\colim M$.
	For a module $M$, we define the modules $K_k M = \bigoplus_{S \in \binom{[n]}{k}} \Colim_S M$ for each $k$.
	If $S \subseteq S'$, then there is a canonical morphism $\Colim_{S} M \to \Colim_{S'} M$.
	For a free module $F(z)$, these assemble to an exact sequence
	\begin{equation}
		\label{proof:calabi-yau:les-1}
		\quad 0 \to F(z) \to K_1 F(z) \to \dotsb \to K_n F(z) \to I(z)\Shift{\epsilon} \to 0,
	\end{equation}
	called the \emph{Koszul complex} of $F(z)$.
	The last morphism is the canonical morphism $K_n F(z) = \Colim F(z) = \Lim I(z)\Shift{\epsilon} \to I(z)\Shift{\epsilon}$.
	Let $F_\bullet$ be a bounded complex of free modules.
	Using $\nu F(z) = I(z)$, we get an exact sequence
	\begin{equation}
		\label{proof:calabi-yau:les-2}
		K_\bullet \colon \quad 0 \to F_\bullet \to K_1 F_\bullet \to \dotsb \to K_n F_\bullet \to \nu F_\bullet\Shift{\epsilon} \to 0
	\end{equation}
	of chain complexes, given by taking a shifted copy of \cref{proof:calabi-yau:les-1} for every $F(z)$ in $F_\bullet$.
	We unsplice~\eqref{proof:calabi-yau:les-2} into short exact sequences
	\begin{equation}
		\label{proof:calabi-yau:ses}
		\begin{tikzcd}[ampersand replacement=\&, row sep={0.8cm,between origins}, column sep={1cm,between origins}]
			\& \& \& 0 \ar[dr] \& \& 0 \& 0 \ar[dr] \& \& \& \&  \&[0.4cm] \\
			\& \& \& \& U^{(1)}_\bullet \ar[dr] \ar[ur] \& \& \& U^{(n-1)}_\bullet \ar[dr] \\
			0 \ar[r] \&
			F_\bullet \ar[rr] \ar[dr, equal] \& \&
			K_1 F_\bullet \ar[rr] \ar[ur] \&\&
			K_2 F_\bullet \ar[rrr, "\textstyle\dotsb" description] \ar[dr] \&\&\&
			K_n F_\bullet \ar[rr] \ar[dr] \& \&
			\nu F_\bullet\Shift{\epsilon} \ar[r] \&
			0
			\\
			\& \& F_\bullet \ar[ur] \& \& \& \&
			U^{(2)}_\bullet \ar[dr] \& \& \& \nu F_\bullet\Shift{\epsilon} \ar[ur, equal] \ar[dr]
			\\
			\& 0  \ar[ur]  \& \& \& \& \& \& 0 \& \& \& 0
		\end{tikzcd}
	\end{equation}
	with chain complexes $U^{(k)}_\bullet$ for each $k$.
	Each of these short exact sequences gives a triangle in the derived category $\Db(\Pers{\Z^n})$ \cite[\S 10.4.9]{Weibel:2003}.
	We obtain connecting homomorphisms
	\begin{align*}
		\partial^{(1)}\colon  U^{(1)}_\bullet[1] &\to F_\bullet,&
		\partial^{(2)}\colon  U^{(2)}_\bullet[1] &\to U^{(1)}_\bullet,&
		&\cdots &
		\partial^{(n-1)}\colon \nu F_\bullet \Shift{\epsilon}[1] &\to U^{(n-1)}_\bullet
	\end{align*}
	in $\Db(\Pers{\Z^n})$.
	These fit into the long exact sequences
	\begin{equation}
		\label{proof:calabi-yau:les}
		\begin{aligned}
			\dotsb \to
			H_{d+1}(K_1 F_\bullet) \to
			H_{d+1}(U^1_\bullet) &\xrightarrow{\mathmakebox[0.8cm][c]{\partial^{(1)}}}
			H_{d}(F_\bullet) \to
			H_{d}(K_1 F_\bullet) \to
			\dotsb,
			\\
			\dotsb \to
			H_{d+2}(K_2 F_\bullet) \to
			H_{d+2}(U^2_\bullet)
			&\xrightarrow{\mathmakebox[0.8cm][c]{\partial^{(2)}}}
			H_{d+1}(U^{(1)}_\bullet) \to
			H_{d+1}(K_2 F_\bullet) \to
			\dotsb,
			\\
			&\mathrel{\mathmakebox[\widthof{$\xrightarrow{\mathmakebox[0.8cm][c]{}}$}]{\vdots}} \\
			\dotsb \to
			H_{d+n}(K_n F_\bullet) \to
			H_{d+n}(\nu F_\bullet\Shift{\epsilon})
			&\xrightarrow{\mathmakebox[0.8cm][c]{\partial^{(n-1)}}}
			H_{d+n-1}(U^{(n-1)}_\bullet) \to
			H_{d+n-1}(K_n F_\bullet) \to \dotsb,
		\end{aligned}
	\end{equation}
	induced by the short exact sequences \cref{proof:calabi-yau:ses}.
	Since $H_d(F_\bullet)$ is of finite total dimension for all $d$, we have $\Colim_S H_d(F_\bullet) = 0$ if $\abs{S} > 0$.
	The functor $\Colim_S$ is exact for all $S$ because it is a directed colimit.
	In particular,
	\(
		\textstyle
		H_d(K_k F_\bullet)
		= H_d\bigl({\bigoplus_{\abs{S}=k}} \Colim_S F_\bullet\bigr)
		= 0
	\)
	for all $k > 0$.
	Therefore, the long exact sequences \cref{proof:calabi-yau:les}
	show that all connecting homomorphisms~$\partial^{(k)}$ are quasi-isomorphisms.
	Thus, $\partial^{(1)} \circ \dotsb \circ \partial^{(n-1)}\colon \nu F_\bullet [n] \Shift{\epsilon} \longrightarrow F_\bullet$ is a quasi-isomorphism.
\end{proof}

\begin{corollary}
	\label{thm:calabi-yau for modules}
	Let $M \in \Pers{\Z^n}$ be of finite total dimension.
	\begin{enumerate}
		\item If $F_\bullet$ is a free resolution of $M$, then $\nu F_\bullet[n]\Shift{\epsilon}$ is an injective resolution of $M$.
		\item If $I^\bullet$ is an injective resolution of $M$, then $\nu^{-1} I^\bullet[-n]\Shift{-\epsilon}$ is a free resolution of $M$.
	\end{enumerate}
\end{corollary}

\begin{proof}[Proof of \cref{thm:UQ-LQ-isomorphism}]
	\label{proof:UQ-LQ-isomorphism}
	With $N^\bullet(K_*) = (\nu C_\bullet(K_*) \Shift\epsilon)^*$, \cref{thm:calabi-yau} gives
	\begin{equation*}
		H^{d+n}(N^\bullet(K_*))^*
		\cong H_d(N^\bullet(K_*)^*[n]) 
		\cong H_d(\nu C_\bullet(K_*) \Shift{\epsilon} [n]) 
		\cong H_d(C_\bullet(K_*)) 
		= H_d(K_*).\qedhere
	\end{equation*}
\end{proof}

\begin{lemma}
	\label{thm:graded matrices injective modules}
	A graded matrix $M$ represents a morphism $\bigoplus_j I(\cg^M_j) \to \bigoplus_i I(\rg^M_i)$  \textiff\ $M_{ij} =0$ whenever $\rg^M_i \not \leq \cg^M_j$.
\end{lemma}
\begin{proof}
	This follows from $\Hom(I(z), I(z')) = \begin{smallcases} k & \text{if $z \geq z'$} \\ 0 & \text{otherwise}\end{smallcases}$
\end{proof}
In particular, a graded matrix represents a morphism of free modules \textiff\ it represents a morphism of injective modules (cf.~\cref{thm:graded matrices free modules}).

\begin{lemma}
	\label{thm:nakayama on matrices}
	Let $f\colon \bigoplus_{j =1}^n F(z_j) \to \bigoplus_{i =1}^m F(z'_i)$ be a morphism of free modules represented by the graded matrix $[f]$.
	Then the morphism $\nu f\colon \bigoplus_{j =1}^n I(z_j) \to \bigoplus_{i =1}^m I(z'_i)$ is represented by the same graded matrix $[\nu f] = [f]$.
\end{lemma}

\begin{corollary}
	\label{thm:calabi-yau for matrices}
	For $M \in \Pers{\Z^n}$ of finite total dimension and graded matrices $U_1,\dotsc,U_n$, the following are equivalent:
	\begin{enumerate}
		\item $U_1,\dotsc,U_n$ represent a free resolution $\dotsb 0 \to F_n \xto{U_n} \dotsb \xto{U_1} F_0$ of $M$,
		\item $U_1\Shift{\epsilon},\dotsc,U_n\Shift{\epsilon}$ represent an injective resolution $I^0 \xto{U_n\Shift{\epsilon}} \dotsb \xto{U_1\Shift{\epsilon}} I^n \to 0 \dotsb$ of $M$,
		\item $U_1\Shift{\epsilon}^T,\dotsc,U_n\Shift{\epsilon}^T$ represent a free resolution $\dotsb 0 \to G_n \xto{U_n\Shift{\epsilon}^T} \dotsb \xto{U_1\Shift{\epsilon}^T} G_0$ of $M^*$.
	\end{enumerate}
	In this case $I^q = \nu F_{n-q}\Shift{\epsilon} = G_q^*$ for all $q$.
\end{corollary}

\begin{example}
	\tikzset{module diagram/.append style={x=3mm, y=3mm, baseline=4mm}}
	\newcommand{\coordinatesA}{
		\coordinate (G1)  at=(0.5,2);
		\coordinate (G2)  at=(2,0.5);
		\coordinate (R1)  at=(0.5,3.5);
		\coordinate (R2)  at=(3.5,0.5);
		\coordinate (R3)  at=(2,2);
		\coordinate (S1)  at=(3.5,3.5);
		\DrawGrid[-0.5](0){4}[-0.5](0){4}
		\AxesHomology
	}
	Consider the module
	\[
		M = \begin{tikzpicture}[module diagram]
			\coordinatesA
			\node[generator, at=(G1)] {};
			\node[generator, at=(G2)] {};
			\node[relation, at=(R1)] {};
			\node[relation, at=(R2)] {};
			\node[relation, at=(R3)] {};
			\node[syzygy, at=(S1)] {};
			\filldraw[fill=gray, fill opacity=0.5] ($(R1)-(0,.5)$) -- (G1) -- (R3) -- (G2) -- ($(R2)-(0.5,0)$) -- ($(S1)-(0.5,0.5)$) -- ($(R1)-(0,0.5)$);
		\end{tikzpicture},
	\]
	where $M_z = k$ if $z$ lies in the shaded region, $M_z = 0$ otherwise, and all structure morphisms between non-zero vector spaces of $M$ being identities.
	The first line of the following diagram exhibits $F_\bullet$ as a free resolution of $M$, and the second line exhibits $\nu F_\bullet[2]\Shift{\epsilon}$ as an injective resolution of $M$:
	\begin{gather*}
		\begin{tikzcd}[row sep=0pt, column sep=small, ampersand replacement=\&]
			\&
			0
			\rar \&
			\begin{tikzpicture}[module diagram]
				\coordinatesA
				\FreeModule[violet]{S1}{syzygy}
			\end{tikzpicture}
			\rar["{\Mtx*[r]{1 \\ -1 \\ 1}}"] \&[1em]
			\begin{tikzpicture}[module diagram]
				\coordinatesA
				\FreeModule[red]{R1}{relation}
				\FreeModule[blue]{R2}{relation}
				\FreeModule[violet]{R3}{relation}
			\end{tikzpicture}
			\rar["{\Mtx*[r]{1 & 1 &0 \\ 0 & -1 & 1}}"] \&[2em]
			\begin{tikzpicture}[module diagram]
				\coordinatesA
				\FreeModule[orange]{G1}{generator}
				\FreeModule[magenta]{G2}{generator}
			\end{tikzpicture}
			\rar \&
			\begin{tikzpicture}[module diagram]
				\coordinatesA
				\node[generator, at=(G1)] {};
				\node[generator, at=(G2)] {};
				\node[relation, at=(R1)] {};
				\node[relation, at=(R2)] {};
				\node[relation, at=(R3)] {};
				\node[syzygy, at=(S1)] {};
				\filldraw[fill=gray, fill opacity=0.5] ($(R1)-(0,.5)$) -- (G1) -- (R3) -- (G2) -- ($(R2)-(0.5,0)$) -- ($(S1)-(0.5,0.5)$) -- ($(R1)-(0,0.5)$);
			\end{tikzpicture}
			\rar \&
			0 \\[-2mm]
			\& \& F_2 \& F_1 \& F_0 \& M
			\\[2mm]
			0
			\rar \&
			\begin{tikzpicture}[module diagram]
				\coordinatesA
				\node[generator, at=(G1)] {};
				\node[generator, at=(G2)] {};
				\node[relation, at=(R1)] {};
				\node[relation, at=(R2)] {};
				\node[relation, at=(R3)] {};
				\node[syzygy, at=(S1)] {};
				\filldraw[fill=gray, fill opacity=0.5] ($(R1)-(0,.5)$) -- (G1) -- (R3) -- (G2) -- ($(R2)-(0.5,0)$) -- ($(S1)-(0.5,0.5)$) -- ($(R1)-(0,0.5)$);
			\end{tikzpicture}
			\rar \&
			\begin{tikzpicture}[module diagram]
				\coordinatesA
				\node[syzygy, at=(S1), violet] {};
				\InjectiveModule[violet]{$(S1) - (0.5,0.5)$}
			\end{tikzpicture}
			\rar["{\Mtx*[r]{1 \\ -1 \\ 1}}"] \&
			\begin{tikzpicture}[module diagram]
				\coordinatesA
				\node[red, relation] at (R1) {};
				\node[blue, relation] at (R2) {};
				\node[violet, relation] at (R3) {};
				\InjectiveModule[red]{$(R1)-(0.5,0.5)$}
				\InjectiveModule[blue]{$(R2)-(0.5,0.5)$}
				\InjectiveModule[violet]{$(R3)-(0.5,0.5)$}
			\end{tikzpicture}
			\rar["{\Mtx*[r]{1 & 1 & 0  \\ 0 & -1 & 1}}"] \&
			\begin{tikzpicture}[module diagram]
				\coordinatesA
				\node[orange, generator] at (G1) {};
				\node[magenta, generator] at (G2) {};
				\InjectiveModule[orange]{$(G1)-(0.5,0.5)$}
				\InjectiveModule[magenta]{$(G2)-(0.5,0.5)$}
			\end{tikzpicture}
			\rar \&
			0. \\[-2mm]
			\& M \& \nu F_2\Shift{\epsilon} \& \nu F_1\Shift{\epsilon} \& \nu F_0\Shift{\epsilon}
		\end{tikzcd}
	\end{gather*}
\end{example}

\subsection{Pulling back modules from the colimit}
\label{sec:computing lq-cohomology}
From now on, we consider $\Z^2$-persistence modules only.
It remains to explain how we compute $H^\bullet(N^\bullet(K_*))$.
In principle, this could be done by a procedure analogous to the LW-Algorithm described in \cref{sec:lw-algorithm}:
the horizontal sequence in the commutative diagram%
\begin{equation}
	\label{eq:free resolution cohomology}
	\begin{tikzcd}[ampersand replacement=\&]
		\& N^d(K_*) \ar[dr, "\d^{d+1}", ""' name=x] \dar \\
		0 \rar \& Z^{d+1}(N^\bullet(K_*)) \rar[ "i^{d+1}"', "" name=y]
		\& N^{d+1}(K_*) \rar[ "p^{d+1}"] \& Z^{d+2}(N^\bullet(K_*)) \dar[ "i^{d+2}"] \\
		\& \& \& N^{d+2}(K_*) \ar[dr, "\d^{d+3}"] \ar[from=ul, "\d^{d+2}"']\\
		\& \& \& \& N^{d+3}(K_*)
	\end{tikzcd}
\end{equation}
is a free resolution of $H^{d+2}(N^\bullet(K_*))$,
and we can obtain matrices representing this resolution as described in \cref{sec:lw-algorithm}.
This would, however, involve the coboundary maps $\d^{d+2}$ and $\d^{d+3}$, leading to a very expensive computation, especially for function-Rips bifiltrations.
Instead, we propose a method that computes a free resolution of $H^{d+2}(N^\bullet(K_*))$ from $\d^{d+1}$ only.

For a vector space $V$, denote by $\Delta V$ the persistence modules with components $(\Delta V)_z = V$, such that all structure morphisms of $\Delta V$ are the identity.
Let $\Colim M = \Delta \colim M$.
\begin{definition}
	For a module $M \in \Pers{\ZZ}$ and a vector space $V \subseteq \colim M$,
	we let~$[V]_M \in \Pers{\ZZ}$ be the preimage of $\Delta V$ under the canonical map $\eta_M\colon M \to \Colim M$.
\end{definition}

\begin{lemma}
	\label{thm:pullback of kernels}
	If $f\colon M \to N$ is a morphism and $N$ is free, then $\ker f = [\colim \ker f]_M$.
\end{lemma}
\begin{proof}
	\label{proof:pullback of kernels}
	If $N$ is free, then $\eta_N\colon N \to \Colim N$ is injective.
	For every submodule $L \subseteq M$, we have $L \subseteq \eta_M^{-1}(\Colim L) = [\colim L]_M$, so $\ker f \subseteq [\colim \ker f]_M$.
	It remains to show the other inclusion $[\colim \ker f]_M \subseteq \ker f$.
	Consider the commutative diagram
	\[
	\begin{tikzcd}
		{} [\colim\ker f]_M \ar[drr, hook, "j", bend left=0.4cm, shorten >=5pt] \ar[ddr, "\eta_{[\colim \ker f]_M}"', bend right=0.4cm] &[-15mm]\\[-2mm]
		& \ker f \rar["i"', hook] \dar["\eta_{\ker f}"] \ar[hook, ul]& M \rar["f"] \dar["\eta_M"] & N \dar[hook, "\eta_N"] \\
		& \Colim \ker f \rar["\Colim i"', hook] & \Colim M \rar["\Colim f"'] & \Colim N.
	\end{tikzcd}
	\]
	The functor $\Colim$ is a directed colimit and thus exact.
	Therefore, $\Colim \ker f = \ker \Colim f$.
	This implies $\eta_N \circ f \circ j = \Colim f \circ \Colim i \circ \eta_{[\colim \ker f]_M} = 0$.
	Since $\eta_N$ is injective, we obtain $f \circ j = 0$.
	Therefore, $j$ factors uniquely through $\ker f$.
	This proves the claim.
\end{proof}

The \emph{lexicographic order} $\lex{\preceq}$ and the \emph{colexicographic order} $\colex{\preceq}$ are the total orders on~$\ZZ$
defined as \begin{alignat*}{2}
	(x, y) &\lex\preceq (x', y') &&\ \text{iff either $x < x'$ or $x = x'$ and $y \leq y'$},\\
	(x, y) &\colex\preceq (x', y') &&\ \text{iff either $y < y'$ or $y = y'$ and $x \leq x'$}.
\end{alignat*}
Two grades $z_1, z_2 \in \ZZ$ satisfy $z_1 \leq z_2$ \textiff\ $z_1 \colex{\preceq} z_2$ and $z_1 \lex{\preceq} z_2$.

\begin{definition}
	For $b \in k^m$ and $r \in (\ZZ)^m$, the \emph{lex pivot} of $b$ with respect to $r$, $\lpiv(b)$, is the smallest index $i$  such that $b_i\ne 0$ and $r_i$ takes its maximum value \wrt\ $\lex\preceq$.  The \emph{colex piot}, $\cpiv(b)$, is defined analogously.  For $0\in k^m$, we let $\lpiv(0)=\cpiv(0)=0$.
	\end{definition}

\begin{algorithm}[tbp]
	\KwData{An $m \times n$-matrix $B$ representing a generating set of $V$, $r=(r_1,\ldots,r_m) \in (\ZZ)^m$.}
	\KwResult{A graded $m \times n$-matrix whose nonzero columns represent a basis of $[V]_M$.}
	\caption{%
		Computes a basis of $[V]_M$, where $M = \bigoplus_{i = 1}^m F(r_i)$ and $V \subseteq \colim M$.
	}
	\label{alg:bigraded reduction}
	\Fn{\Bireduce{B}}{
	$p \gets 0 \in [m]^n$\;
	\For{$j = 1,\dotsc,n$}{
		\label{alg:bigraded reduction:column loop 1}
		\While{$i \gets \cpiv(B_j) \neq 0$}{
			\lIf{$p_i = 0$}{$p_i \gets j$; \Break}
			$B_j \gets B_j + B_{p_i}$
		}
	}
	$p \gets 0 \in [m]^n$\;
	\For{$j' = 1,\dotsc,n$}{
		$j \gets j'$\;
		\While{$i \gets \lpiv(B_j) \neq 0$}{
			\lIf{$p_i = 0$}{$p_i \gets j$; \Break}
			\nlset{$(*)$}\label{alg:bigraded reduction:secondary pivot}%
			\lIf{$\cpiv(B_j) < \cpiv(B_{p_i})$}{swap $p_i$ and $j$}
			$B_j \gets B_j + B_{p_i}$
		}
	}
	\Return{$[B]_r$}
	}
\end{algorithm}

\begin{theorem}
	\label{thm:correctness of bireduction algorithm}
	Let $M$ be a free $\ZZ$-persistence module of finite rank with a fixed basis,
	let $V \subseteq \colim M$ be a subspace, and let $B$ be a matrix representing a generating set of $V$.
	Then $[V]_M$ is free, and \cref{alg:bigraded reduction} calculates a graded matrix representing a basis of $[V]_M$.
\end{theorem}
For a tuple $r \in (\ZZ)^m$ and a matrix $M$ with $m$ rows, let $[M]_r$ be the graded matrix with $\rg^{[M]_r}_i = r_i$ and $\cg^{[M]_r}_j = \bigvee_{M_{ij} \neq 0} r_i$.
Then $[M]_r$ has the least possible column grades for which $[M]_r$ represents a map of free modules.
\begin{proof}[Proof of \cref{thm:correctness of bireduction algorithm}]
	Let $m_1, \dotsc, m_s$ be a basis of $M$  and $r = (g(m_1), \dotsc, g(m_s))$.
	Without loss of generality, we assume that $B$ represents a basis of $V$.
	The first for-loop in \cref{alg:bigraded reduction} is a standard reduction scheme.
	In each iteration, the pivot index of one column decreases, so the loop terminates.
	When it does, all columns have distinct colex-pivots.
	During each iteration of the second for-loop, the lex-pivot of a column decreases.
	When it terminates, all columns have distinct lex-pivots.
	During the second loop, line \ref{alg:bigraded reduction:secondary pivot} ensures that no column is added to another column with a smaller colex-pivot.
	Since all columns have distinct colex-pivots after the first for-loop, the colex-pivots of the columns thus do not change during the second for-loop.
	Therefore, when the algorithm terminates, all columns of $B$ have pairwise distinct lex- and colex-pivots.

	Let $A = [B]_r$ for the state of $B$ when the algorithm terminates.
	Then $A$ represents a basis $\alpha_1, \dotsc, \alpha_t$ of a free submodule $N$ of $M$, with $g(\alpha_j) = \cg^A_j$ .
	It remains to show that $N = [V]_M$.
	Since all column operations performed by \cref{alg:bigraded reduction} are invertible, $A$ represents a basis of $V$.
	Therefore, $\colim N = V$, which implies $N \subseteq [V]_M$.
	Let $v \in [V]_M$.
	Then there are unique coefficients $\xi_j$ such that $\eta_M(v) = \sum_{j=1}^{t} \xi_j \eta_M(\alpha_j)$.
	Since $\eta_M$ is injective, $M_{g(v) \leq z}(v) = \sum_{j=1}^{t} \xi_j M_{g(\alpha_j) \leq z} (\alpha_j)$ for all $z \geq g(v) \vee \bigvee_{\xi_j \neq 0} \cg^{A}_j$.
	Since all columns of $A$ have distinct lex- and colex-pivots, $v$ cannot have smaller grade than $\bigvee_{\xi_j \neq 0} \cg^{A}_j$,
	so $v \in N$.
	This proves the claim.
\end{proof}

\subsection{The free resolution of cohomology}
\label{sec:computing cohomology}
Assume $C^\bullet$ is a cochain complex of free modules such that $\colim H^\bullet(C^\bullet) = 0$,
and recall the commutative diagram \eqref{eq:free resolution cohomology}.
A matrix $[\d^{d+1}]$ representing $\d^{d+1}$ is a generating system for $\colim Z^{d+1}(C^\bullet)$,
and \cref{thm:pullback of kernels} states that $Z^{d+1}(C^\bullet) = [\colim Z^{d+1}(C^\bullet)]_{C^{d+1}}$.
Applying \cref{alg:bigraded reduction} to $[\d^{d+1}]$ thus yields a graded matrix $[i^{d+1}]$ representing a basis of $Z^{d+1}(C^\bullet)$.
\begin{lemma}
	\label{thm:extend resolution to right}
	If $0 \to F_2 \xto{f_2} F_1 \xto{f_1} F_0$ is a free resolution of a module of finite total dimension, then $(\nu F_0)^* = \ker (\nu f_2)^*$.
\end{lemma}
\begin{proof}
	The sequence
	\(0 \to F_2 \xto{f_2} F_1 \xto{f_1} F_0\)
	is exact.
	By \cref{thm:calabi-yau}, the sequence
	\(
		\nu F_2 \xto{\nu f_2} \nu F_1 \xto{\nu f_1} \nu F_0 \to 0
	\)
	and therefore also the dual sequence
	\(
		0 \to (\nu F_0)^* \xto{(\nu f_1)^*}  (\nu F_1)^* \xto{(\nu f_2)^*} (\nu F_2)^*
	\)
	are exact.
\end{proof}

Thus, if a matrix $[f_2]$ representing $f_2$ is known, then the matrix $[f_1]^T = [(\nu f_1)^*] $ can be computed by applying the LW-Algorithm (\cref{alg:lw-algorithm}) to $[f_2]^T = [(\nu f_2)^*]$.
In particular, if $H_d(C^\bullet)$ has finite total dimension, then so has $H^{d+2}(N^\bullet(K_*))$,
so \cref{thm:extend resolution to right} can be applied to the free resolution \cref{eq:free resolution cohomology} of $H^{d+2}(N^\bullet(K_*))$.
This shows that $[p^{d+1}]^T$ can be computed by applying the LW-Algorithm to $[i^{d+1}]^T$.

\begin{algorithm}[tbp]
	\caption{Computes a minimal free resolution of $H^\bullet(C^\bullet)$ for a cochain complex $C^\bullet$ of free $\ZZ$-modules, using clearing.}
	\label{alg:cohomology resolution}
	\KwIn{Graded matrices $[\d^\bullet]$ representing $C^\bullet$.}
	\KwOut{Pairs of graded matrices representing a free resolution of $H^d(C^\bullet)$ for $d = 0, 1, \dotsc$.}
	$q \gets \emptyset$\tcp*[r]{pivots for clearing}
	\For{$d = 0, 1, \dotsc$}{
		\lFor(\tcp*[f]{clearing}){$j \in q$}{$[\d^{d+1}]_j \gets 0$}
		$[i^{d+1}] \gets \Bireduce{[\d^{d+1}]}$\;
		$n \gets \# \text{columns of } [i^{d+1}]$\;
		$q \gets \Set{\piv [i^{d+1}]_1, \dotsc, \piv [i^{d+1}]_n}$\;
		$[i^{d+1}]^T, [p^{d+1}]^T \gets \MgsWithKer{[i^{d+1}]^T}$ \tcp*[r]{See \cref{alg:ker-mgs}}
		\Yield $\MinimizeChain{[i^{d+1}], [p^{d+1}]}$ \tcp*[r]{\MFR\ of $H^{d+2}(\nu C^\bullet\Shift{\epsilon}) \cong H^d(C^\bullet)$}
	}
\end{algorithm}

\begin{corollary}
	\label{thm:cohomology computation}
	Let $C^\bullet$ be a cochain complex of free modules such that $\dim H^\bullet(C^\bullet)$ is finite.
	Then \cref{alg:cohomology resolution} computes free resolutions of $H^d(C^\bullet)$.
\end{corollary}

\begin{remark}[Clearing]
	In general, $[\d^{d+1}]$ is not injective.
	As in one-parameter persistent cohomology,
	the first loop in \cref{alg:bigraded reduction} spends a significant amount of time on reducing the columns of $[\d^{d+1}]$ that are eventually reduced to zero.
	The computation can be accelerated considerably by using the pivots of the reduced matrix $[\d^d]$ to implement a clearing scheme
	before invoking \cref{alg:bigraded reduction}.
	This is implemented in \cref{alg:cohomology resolution}.
\end{remark}

\section{Experiments}
We have implemented our cohomology algorithm in C++ \cite{Lenzen:2023}.
We have also implemented the alogrithm \cite{FugacciKerberEtAl:2023} used in \texttt{mpfree} \cite{Kerber:2021}, in order to vary the implementation details.
Where applicable, the run time of our clone is similar to the one of \texttt{mpfree}.
We have run our implementation to compute \MFR{}s of the \PH{} of various function-Rips bifiltrations.

\subsection{Setup}
All computations are done with coefficients in $k = \mathbf{F}_2$.
Matrix columns are implemented as binary heaps \cite{BauerKerberEtAl:2017}.
Our code also implements an alternative representation of columns as dynamically allocated arrays.
We have run our code on a MacBook Pro 2017 with a 2.3 GHz Dual-Core Intel Core i5 and 16GB RAM.
The code is compiled using clang++ 15.0.7.  Each instance of our program may run four threads in parallel.

The run time of the homology algorithm for minimal presentation computation  \cite{FugacciKerberEtAl:2023} is dominated by chunk preprocessing and the LW-Algorithm.
While it is standard to implement chunk preprocessing in an embarrassingly parallel way, no way is known to parallelize the LW-Algorithm.
While the minimization step in \cref{alg:cohomology resolution} is parallelized in our implementation, the bigraded reduction (\cref{alg:bigraded reduction}) is not, although we hypothesize it could be parallelized analogously to \cite{MorozovNigmetov:2020} or \cite{BauerKerberEtAl:2014}.
We found that the minimization is not a performance bottleneck of our algorithm, so one would expect similar performance on a single core.

\subparagraph*{Datasets}
We have generated point clouds $S$ by sampling $n$-spheres $S^n$, $n$-tori $S^1 \times \dotsb \times S^1$ and orthogonal groups $O(n)$.
Additionally, we use some of the point clouds from \cite{OtterPorterEtAl:2017}.
To each point $p \in S$, we associate the value $\sum_{q \in S \setminus \{p\}} \exp\bigl(-\frac{d(p, q)^2}{2\sigma^2}\bigr)$ for a manually chosen parameter~$\sigma$.
These values and a distance matrix are written to a file, from which the program generates the coboundary matrices of the associated full function-Rips bifiltration.

\subparagraph*{Chunk preprocessing}
The LW-Algorithm works most efficiently if combined with chunk preprocessing \cite{FugacciKerber:2019a,FugacciKerberEtAl:2023}; this is the approach implemented in \texttt{mpfree}.
Chunk preprocessing (\cref{algo:chain-chunk}) applies a certain column operation scheme to the matrices representing the chain complex $C_\bullet$.
As an alternative, we propose to manipulate $C_\bullet(K_*)$ by row operations.
Equivalently, one can see this procedure as column operations on the matrices representing the cochain complex $\nu C^\bullet(K_*)$;
hence, we refer to this approach as \emph{cochain complex chunk preprocessing}; see \cref{algo:cochain-chunk}.

\subparagraph*{Coning off}
To ensure that $H^{d+2}(N^\bullet(K_*)) \cong H_d(K_*)^*$, \cref{thm:cohomology computation} requires that $C_\bullet$ has homology of finite total dimension.
Therefore, our implementation offers the ability to cone off the complex as follows.
Let $C_\bullet$ be a chain complex of free $\ZZ$-persistence modules.
The assignment $C'_d\colon y \mapsto \colim_{x \in \Z} (C_d)_{xy}$ defines a chain complex $C'_\bullet$ of free $\Z$-persistence modules.
We compute the barcode of $H_d(C'_\bullet)$ using the cohomological standard algorithm with clearing (\cref{algo:1D persistence clearing}).
This can be used to implement a clearing mechanism in the homological standard algorithm,
which we use to compute representatives for the homology classes in the barcode of $H_d(C'_\bullet)$.
Let $y_0$ such that $y_0 \geq g_y(\sigma)$ for all $\sigma$, where $g(\sigma) = (g_x(\sigma), g_y(\sigma))$.
Let $\hat{C}_\bullet = C_\bullet$.
For every bar $(b, d)$ of $H_d(C'_\bullet)$ of non-zero length represented by a $q$-cycle $c \in C'_q$ of grade $g(c) = b$,
we add a basis element $\hat{c}$ of grade $g(\hat{c}) = (b, y_0)$ to $\hat{C}_{q+1}$ with $\partial_{q+1}(\hat{c}) = c$.
If $d < \infty$, then $c$ bounds a chain $c'$ with $g(c') = d$,
and we add a basis element $\hat{c}'$ of grade $g(\hat{c}') = (d, y_0)$ to $\hat{C}_{q+2}$ with $\partial_{q+2}(\hat{c}') = c' - \hat{c}$.
The resulting chain complex $\hat{C}_\bullet$ satisfies $H_\bullet(\hat{C}_\bullet)_{xy} = 0$ for $y \geq y_0$.
If not stated otherwise, cohomology computation is done with this preprocessing applied to the density parameter.

\subparagraph{Sparsification}
We observe that the second for-loop in \cref{alg:bigraded reduction} runs considerably longer than the first.
The loop also increases the matrix density, which many incur a high cost on the subsequent steps in \cref{alg:cohomology resolution}.
For an interpretation, see \cref{rmk:run time loops}.
As a remedy, we have added a step that decreases the sparsity of the matrix using row operations that are compatible with the column sparse matrix implementation.
Specifically, if a row contains only a single entry, any row addition from this row to another affects only a single entry.
Therefore, an entry in a row with grade $g$ can be eliminated directly if there is a row with grade $g' \geq g$ containing only a single entry in the same column;
see \cref{algo:sparsification}.
All cohomology computation run times are reported with this sparsification scheme applied.

\subsection{Results}
\begin{table}
	\caption{%
		Run times (in milliseconds) comparing our implementation of \cite{LesnickWright:2022,FugacciKerberEtAl:2023} (including chunk preprocessing) and our cohomology algorithm,
		applied a density-Rips filtration on $300$ vertices ($d=1$), $100$ vertices ($d=2$) and $60$ vertices ($d=3$).
		RSS is peak resident memory as measured by \texttt{time}.
		Speed up is the run time of the homology computation (including chunk preprocessing), divided by the run time of the cohomology algorithm.
		The program has been killed after exceeding five minutes of run time.
	}
	\label{tab:results compact}
	\centering
	\begin{tabular}{rr|rrrr|rrr}
		\toprule
		$d$ & sample & chunk   & $H_d$  & sum     & RSS        & $H^{d+2}$ & RSS       & speedup \\ \midrule
		1   & c. elegans & 5,457   & 40,841 & 46,298  & 6,423,600  & 119,444   & 6,526,976 & 0.39    \\
		& 2-torus    & 11,480  & 19,875 & 31,355  & 6,620,404  & 5,032     & 2,827,912 & 6.23    \\
		& 4-torus    & 6,342   & 28,627 & 34,969  & 5,916,816  & 50,607    & 3,384,472 & 0.69    \\
		& dragon     & 9,721   & 18,489 & 28,210  & 5,774,492  & 5,064     & 2,829,620 & 5.57    \\
		& 2-sphere   & 8,657   & 29,180 & 37,837  & 6,421,312  & 25,021    & 4,098,268 & 1.51    \\
		& 4-sphere   & 7,699   & 33,619 & 41,318  & 6,642,260  & 47,355    & 7,154,632 & 0.87    \\
		& $O(3)$         & 7,023   & 33,874 & 40,897  & 6,315,708  & 42,702    & 3,816,124 & 0.96    \\ \midrule
		2   & c. elegans & 28,583  & 7,484  & 36,067  & 4,428,440  & 5,655     & 2,371,324 & 6.38    \\
		& 2-torus    & 39,630  & 2,191  & 41,821  & 3,216,372  & 5,054     & 2,334,712 & 8.27    \\
		& 4-torus    & 33,788  & 19,875 & 53,663  & 5,538,232  & 5,969     & 2,425,136 & 8.99    \\
		& dragon     & 19,023  & 2,379  & 21,402  & 2,557,124  & 5,188     & 2,367,488 & 4.13    \\
		& 2-sphere   & 32,611  & 12,416 & 45,027  & 5,099,604  & 6,417     & 2,426,924 & 7.02    \\
		& 4-sphere   & 29,272  & 25,357 & 54,629  & 6,039,576  & 8,637     & 2,505,664 & 6.32    \\
		& $O(3)$         & 31,780  & 29,123 & 60,903  & 6,654,692  & 6,796     & 2,445,996 & 8.96    \\ \midrule
		3   & c. elegans & 38,349  & 2,393  & 40,742  & 3,515,708  & 8,984     & 4,227,820 & 4.53    \\
		& 2-torus    & >300,000 & – & – & 7,141,648  & 11,725    & 5,072,192 & >25.59   \\
		& 4-torus    & >300,000 & – & – & 10,843,356 & 9,930     & 4,358,580 & >30.21   \\
		& dragon     & 67,463  & 2,334   & 69,797  & 6,666,732  & 9,782     & 4,900,800 & 7.14    \\
		& 2-sphere   & 59,385  & 3,110   & 62,495  & 4,280,112  & 9,051     & 4,185,036 & 6.90    \\
		& 4-sphere   & 92,365  & 8,577  & 100,942 & 5,526,344  & 8,818     & 4,197,636 & 11.45   \\
		& $O(3)$         & 204,263 & 25,284 & 229,547 & 7,966,600  & 10,851    & 4,365,864 & 21.15   \\ \bottomrule
	\end{tabular}
\end{table}
An overview of the results is given in \cref{tab:results compact}.
We report only the time needed to compute the \MFR\ and, if applicable, to apply the chunk preprocessing.
We do not report the time necessary to set up the (co)boundary matrices.
In all cases with $d \geq 2$, computing $H^{d+2}(N^\bullet)$ (without chunk preprocessing) was faster than computing $H_d$ with chunk preprocessing.
Our cohomology approach does not benefit from chunk preprocessing.
The speedup of the cohomology approach increases with dimension.
For two instances, computation of $H_d(K_*)$ did not terminate within five minutes, while computing $H^{d+2}(N^\bullet(K_*))$ was no problem.
We also observe that the cohomology algorithm uses less memory for almost all instances with $d \geq 2$.

\subparagraph{Matrix representations}
The efficiency of the LW-Algorithm and of chunk preprocessing does not vary very much depending on the matrix implementation; see \cref{tab:results compact,tab:results vectors}.
In contrast, our cohomology algorithm runs faster in the implementation with binary heaps.
We observe that the vector based implementations generally use less memory than the heap based ones.
This happens because, in contrast to vectors, heaps may contain multiple entries for the same row index.

\subparagraph{Cochain chunk preprocessing}
For the homology computation, the cochain complex chunk preprocessing described above often is more efficient than chunk preprocessing if combined with the heap implementation of matrix columns, see \cref{tab:cochain complex chunk}.
This is true in particular for higher homology dimensions.
If combined with vector-based matrices, cochain chunk preprocessing is less efficient than conventional chunk preprocessing in almost all cases, and does not terminate at all within five minutes.

\bibliographystyle{plainurl}%
\bibliography{submission.bib}

\iffullversion

\newpage

\appendix

\begin{table}
	\caption{%
		Same as \cref{tab:results compact}, but with matrices in vector format.
	}
	\label{tab:results vectors}
	\centering
	\begin{tabular}{rr|rrrr|rrr}
		\toprule
		$d$ & sample     & chunk   & $H_d$  & sum     & RSS       & $H^{d+2}$ & RSS          & speedup \\ \midrule
		1   & c. elegans & 7,148   & 78,592 & 85,740  & 4,576,160 & 207,787   & 4,795,428 & 0.41    \\
		& 2-torus    & 14,381  & 25,010 & 39,391  & 5,318,176 & 5,797     & 2,997,160 & 6.80    \\
		& 4-torus    & 7,028   & 39,563 & 46,591  & 4,538,204 & 17,483    & 3,149,484 & 2.66    \\
		& dragon     & 11,493  & 23,769 & 35,262  & 4,791,380 & 6,185     & 3,006,328 & 5.70    \\
		& 2-sphere   & 10,420  & 47,390 & 57,810  & 5,001,668 & 157,839   & 3,724,396 & 0.37    \\
		& 4-sphere   & 8,677   & 42,742 & 51,419  & 4,989,432 & >300.000 & –       & <0.17          \\
		& $O(3)$         & 8,397   & 53,641 & 62,038  & 4,874,744 & 85,476    & 3,548,940 & 0.73    \\ \midrule
		2   & c. elegans & 34,822  & 5,896  & 40,718  & 3,263,528 & 4,968     & 2,545,052 & 8.20    \\
		& 2-torus    & 22,772  & 1,830  & 24,602  & 2,026,840 & 5,679     & 2,527,568 & 4.33    \\
		& 4-torus    & 35,221  & 20,224 & 55,445  & 4,387,188 & 7,937     & 2,505,016 & 6.99    \\
		& dragon     & 24,914  & 2,338  & 27,252  & 2,057,320 & 4,979     & 2,511,392 & 5.47    \\
		& 2-sphere   & 35,732  & 8,952  & 44,684  & 4,122,656 & 14,455    & 2,546,872 & 3.09    \\
		& 4-sphere   & 36,476  & 24,779 & 61,255  & 4,792,056 & 22,485    & 2,566,200 & 2.72    \\
		& $O(3)$         & 35,287  & 21,767 & 57,054  & 4,482,248 & 9,477     & 2,525,464 & 6.02    \\ \midrule
		3   & c. elegans & 31,102  & 2,023  & 33,125  & 2,949,772 & 14,885    & 3,715,832 & 2.23    \\
		& 2-torus    & 35,045  & 2,460  & 37,505  & 3,021,448 & 10,924    & 4,623,784 & 3.43    \\
		& 4-torus    & 190,647 & 5,444  & 196,091 & 4,083,284 & 14,579    & 4,667,328 & 13.45   \\
		& dragon     & 33,626  & 2,029  & 35,655  & 2,997,148 & 10,564    & 4,602,516 & 3.38    \\
		& 2-sphere   & 127,959 & 3,040  & 130,999 & 3,328,716 & 10,436    & 4,610,180 & 12.55   \\
		& 4-sphere   & 172,314 & 7,343  & 179,657 & 4,506,468 & 12,118    & 4,665,124 & 14.83   \\
		& $O(3)$         & 222,455 & 9,461  & 231,916 & 4,795,464 & 13,840    & 4,683,228 & 16.76   \\ \bottomrule
	\end{tabular}
\end{table}
\begin{table}
	\caption{%
		Cochain complex chunk preprocessing is faster than chunk preprocessing if implemented with heap-based matrices.
		Speedup is the run time of homology with chunk preprocessing (see \cref{tab:results compact,tab:results vectors}),
		divided by the run time of homology with cochain complex chunk preprocessing.
	}
	\label{tab:cochain complex chunk}
	\centering
	\newcommand{\specialcell}[1]{\begin{tabular}[t]{@{}c@{}}#1\end{tabular}}
	\makebox[\linewidth][c]{
		\begin{tabular}{rr|rrrr|rrrr}
			\toprule
			&            & \multicolumn{4}{c|}{heaps}                 & \multicolumn{4}{c}{vectors}                \\
			$d$ & sample     & \specialcell{cochain\\chunk} & $H_d$  & sum    & speedup & \specialcell{cochain\\chunk} & $H_d$  & sum     & speedup \\\midrule
			1   & c. elegans & 2,994         & 41,158 & 44,152 & 1.05    & 6,771         & 82,732 & 89,503  & 0.96    \\
			& 2-torus    & 7,646         & 19,766 & 27,412 & 1.14    & 49,376        & 23,064 & 72,440  & 0.54    \\
			& 4-torus    & 3,662         & 28,792 & 32,454 & 1.08    & 7,770         & 41,340 & 49,110  & 0.95    \\
			& dragon     & 6,793         & 18,512 & 25,305 & 1.11    & 52,949        & 23,664 & 76,613  & 0.46    \\
			& 2-sphere   & 4,868         & 28,967 & 33,835 & 1.12    & 14,583        & 45,654 & 60,237  & 0.96    \\
			& 4-sphere   & 4,153         & 33,685 & 37,838 & 1.09    & 11,393        & 44,709 & 56,102  & 0.92    \\
			& $O(3)$         & 3,902         & 34,299 & 38,201 & 1.07    & 8,798         & 48,319 & 57,117  & 1.09    \\\midrule
			2   & c. elegans & 13,749        & 7,431  & 21,180 & 1.70    & >300,000      & –      & –       & <0.14       \\
			& 2-torus    & 4,937         & 2,218  & 7,155  & 5.85    & 100,607       & 2,009  & 102,616 & 0.24    \\
			& 4-torus    & 12,265        & 19,583 & 31,848 & 1.68    & 154,250       & 21,376 & 175,626 & 0.32    \\
			& dragon     & 4,202         & 2,259  & 6,461  & 3.31    & 160,389       & 4,481  & 164,870 & 0.17    \\
			& 2-sphere   & 13,006        & 11,297 & 24,303 & 1.85    & 203,685       & 10,390 & 214,075 & 0.21    \\
			& 4-sphere   & 14,520        & 25,726 & 40,246 & 1.36    & 177,485       & 23,683 & 201,168 & 0.30    \\
			& $O(3)$         & 12,800        & 27,704 & 40,504 & 1.50    & 180,528       & 29,686 & 210,214 & 0.27    \\\midrule
			3   & c. elegans & 2,078         & 2,033  & 4,111  & 9.91    & 38,301        & 1,981  & 40,282  & 0.82    \\
			& 2-torus    & 9,564         & 1,793  & 11,357 & 26.42   & 5,053         & 1,899  & 6,952   & 5.39    \\
			& 4-torus    & 33,826        & 9,901  & 43,727 & 6.86    & >300,000      & –      & –       & <0.65       \\
			& dragon     & 1,229         & 2,172  & 3,401  & 20.52   & 3,212         & 2,255  & 5,467   & 6.52    \\
			& 2-sphere   & 7,351         & 3,757  & 11,108 & 5.63    & 171,650       & 3,211  & 174,861 & 0.75    \\
			& 4-sphere   & 25,289        & 13,386 & 38,675 & 2.61    & >300,000      & –      & –       & <0.60       \\
			& $O(3)$         & 38,122        & 17,841 & 55,963 & 4.10    & >300,000      & –      & –       & <0.77      \\\bottomrule
		\end{tabular}
	}
\end{table}

\section{Further remarks}
\begin{remark}[Minimization]
	\label{rmk:minimization}
	Recall that a free resolution (or, more generally, a complex of free modules) $F_\bullet$ is called \emph{minimal} if it does not contain any homological ball;
	equivalently, all matrices $[\partial^F_\bullet]$ representing it satisfy $[\partial^F_\bullet]_{ij} = 0$ unless $\rg^{[\partial^F_\bullet]}_i < \cg^{[\partial^F_\bullet]}_j$; cf.\ \cref{thm:graded matrices free modules}.
	Assume that
	\[
		\begin{tikzcd}[ampersand replacement=\&, row sep=small]
			\dotsb \rar \& F_{d+1} \rar["\partial^F_{d+1}"] \dar[equal] \& F_d \rar["\partial^F_d"] \dar[equal] \& F_{d-1} \rar["\partial^F_{d-1}"] \dar[equal] \& F_{d-2} \dar[equal] \rar \& \dotsb \\
			\dotsb \rar \& A \rar[swap]{\Mtx{a\\b}} \& B \oplus C \rar[swap]{\Mtx{c & d \\ e & f}} \& D \oplus E \rar[swap]{(g, h)} \& G \rar[swap] \& \dotsb
		\end{tikzcd}
	\]
	for free modules $A,\dotsc, G$.
	If $c$ is invertible, then the summand $\dotsb \to 0 \to C \to D \to 0 \to \dotsb$ of $F_\bullet$ is isomorphic to a ball.
	In this case, first vertical map in the commutative diagram
	\begin{equation}
		\label{eq:minimize}
		\begin{tikzcd}[ampersand replacement=\&]
			\dotsb \rar \&
			A \rar["\Mtx{a \\ b}"] \dar[equal] \&[1cm]
			B \oplus C \rar["\Mtx{c & d \\ e & f}"] \dar["{\Mtx[\big]{c & 0 \\ e & f-ec^{-1}d}}"]
			\&[1cm]
			D \oplus E \rar["{(g, h)}"] \dar[equal] \&
			G \rar["\Mtx{g & h}"] \dar[equal] \&
			\dotsb \\
			\dotsb \rar \&
			A \rar["{\Mtx[\big]{a+dc^{-1}b \\ b}}"'] \dar[equal] \&
			B \oplus C' \rar["{\Mtx[\big]{c & 0 \\ e & f-ec^{-1}d}}"'] \dar["{(0, 1)}"'] \&
			D \oplus E \rar["{(g, h)}"] \dar["{(1, -ec^{-1})}"] \&
			G \dar[equal] \rar \&
			\dotsb \\
			\dotsb \rar \&
			A \rar["b"'] \&
			C' \rar["{f-ec^{-1}d}"'] \&
			E \rar["h"'] \&
			G \rar \&
			\dotsb.
		\end{tikzcd}
	\end{equation}
	is an isomorphism of complexes, and the second is a quasi-isomorphism.
	The entry $f - ec^{-1}d$ can be computed by performing column operations on $\Mtx{c & d \\ e & f}$.
	Spliting $d$-balls from $F_\bullet$ thus can be achieved in the following steps:
	\begin{enumerate}
		\item Performing invertible column operations, bring $[\partial^F_d]$ into the form $\Mtx[\big]{c & 0 \\ e & f-ec^{-1}d}$ for a maximal invertible matrix $c$.
		 Delete the rows of $\Mtx[\big]{c & 0 \\ e & f-ec^{-1}d}$ corresponding to $B$ and $D$;
			This is done by \cref{algo:minimize}.
		\item Delete rows of $[\partial^F_{d+1}]$ and columns of $[\partial^F_{d-1}]$ corresponding to $B$ and $D$, respectively.
	\end{enumerate}
	In particular, it is not necessary to perform any row operations on $[\partial^F_{d+1}]$.
	\Cref{algo:minimize} returns the matrix $f-ec^{-1}d$, together with the row indices $r$ and $c$ of the rows and columns of $[\partial^F_{d+1}]$ corresponding to $C$ and $E$.
	To split off balls from $F_\bullet$ in all dimensions, steps 1 and 2 have to be applied to all matrices $[\partial^F_1], [\partial^F_2], \dotsc$ subsequently.
	this is done by \cref{algo:chain-chunk} for chain complexes and \cref{algo:cochain-chunk} for cochain complexes.
\end{remark}

\begin{remark}[run time of the loops in \cref{alg:bigraded reduction}]
	\label{rmk:run time loops}
	The second for loop in \cref{alg:bigraded reduction} takes much more time than the first, even if the two are exchanged.
	To explain this, consider the graded matrix $D$ passed to \cref{alg:bigraded reduction}.
	Initially, rows and columns of $D$ are ordered colexicographically \wrt\ their grade.
	This order makes $D$ represent the coboundary morphism of a $\Z$-filtration, and we know that the standard algorithm (which the first for-loop is) is efficient on filtered complexes.
	Essentially, this is because (co)boundary matrices arising from filtered complexes are almost in echelon form already.
	Assume that $D'$ is the result of the first for-loop.
	Let $\rho, \tau$ be permutations such that $\rho D' \sigma$ has rows lexicographically sorted and columns sorted \wrt\ $\cpiv$.
	The second for-loop then is just the standard algorithm applied to $\rho D' \sigma$.
	One checks that $\rho D'\sigma$ cannot come from a non-trivial filtration, which explains the long run time.
\end{remark}

\renewcommand{\bottomfraction}{0.9}
\renewcommand{\floatpagefraction}{1}
\renewcommand{\textfraction}{0.01}

\raggedbottom
\section{Other algorithms}
\noindent
\begin{minipage}{\linewidth}
\begin{algorithm}[H]
	\caption{Standard algorithm for one-parameter persistent relative cohomology with clearing. If $\colim H^\bullet(K, K_*) = 0$, there are no infinite bars in the barcode.}
	\KwIn{Graded matrices $[\d^1], [\d^2],\dotsc$ with ascending rows and columns \wrt\ grade, representing a chain complex $C^\bullet(K, K_*)$ with $\colim H^\bullet(K, K_*) = 0$.}
	\KwOut{barcode of $H^\bullet(K_*)$.}
	$p' \gets \emptyset$\;
	$n \gets \#\mathrm{columns}([\d^{d+1}])$\;
	\For{$d = 0, 1, \dotsc$}{
		$p \gets 0 \in \N^{n}$\;
		\For{$j = 1,\dotsc, n$}{
			\lIf(\tcp*[f]{clearing}){$p'_j \neq 0$}{%
				$[\d^{d+1}]_j \gets 0$;
				\Continue%
			}
			\lWhile{$p_i \neq 0$ for $i = \piv [\d^{d+1}]_j$}{%
				$[\d^{d+1}]_j \gets [\d^{d+1}]_j + [\d^{d+1}]_{p_i}$%
			}
			$p_i \gets j$\;
			\Yield $(\rg^{[\d^{d+1}]}_i, \cg^{[\d^{d+1}]}_j)$\;
		}
		$p' \gets p$\;
	}
\label{algo:1D persistence clearing}
\end{algorithm}
\end{minipage}

\noindent
\begin{minipage}{\linewidth}
\begin{algorithm}[H]
	\caption{The LW-Algorithm \cite{KerberRolle:2021, LesnickWright:2022} computes a basis of the kernel of morphism of free $\ZZ$-persistence modules.}
	\KwIn{A graded $m \times n$-matrix $M$ representing a morphism of free $\ZZ$-persistence modules.}
	\KwResult{A graded matrix $K$ representing $\ker M$}
	\label{alg:lw-algorithm}
	\Fn{\Ker{$M$}}{
		$p \gets 0 \in \N^m$ \tcp*[r]{pivot row to column assignment}
		$Q \gets \Set{ (\cg^M_j, j) }$ as priority queue with $(z, j) \leq (z', j')$ if $z \lex\preceq z'$ or $z = z'$ and $j \leq j'$\;
		$V \gets $ $n \times n$-unit matrix \tcp*[r]{Reduction matrix}
		$K \gets $ empty graded $n \times 0$-matrix with $\rg^K = \cg^M$ \;
		\While{$Q \neq \emptyset$}{
			$(z, j) \gets \PopHeap{Q}$\;
			\Forever{}{
				$i \gets \piv M_j$\;
				\lIf{$i = 0$}{append $V_j$ to $K$ and $z$ to $\cg^K$; \Break}
				\lElseIf{$p_{i} = 0$}{$p_i \gets j$; \Break}
				\lElseIf{$\cg^M_{p_i} \not\leq z$}{%
					\PushHeap{Q, (\cg^M_{p_i} \vee z)};
					$p_i \gets j$;
					\Break
				}
				\Else{
					$M_j \gets M_j + M_{p_i}$\;
					$V_j \gets V_j + V_{p_i}$\;
				}
			}
		}
		\Return{$K$}\;
	}
\end{algorithm}
\end{minipage}

\noindent
\begin{minipage}{\linewidth}
\begin{algorithm}[H]
	\caption{A variant of \cref{alg:lw-algorithm} that computes a minimal generating system of the image of a morphism of free $\ZZ$-persistence modules, together with its kernel.}
	\KwIn{A graded $m \times n$-matrix $M$ representing a morphism of free $\ZZ$-persistence modules.}
	\KwResult{Graded matrices $M'$ representing a minimal generating system of $\im M$ and $K'$ representing a basis of $\ker M'$.}
	\label{alg:ker-mgs}
	\Fn{\MgsWithKer{$M$}}{
		$p \gets 0 \in \N^m$ \tcp*[r]{pivot row to column assignment}
		$Q \gets \Set{ (\cg^M_j, j) }$ as priority queue with $(z, j) \leq (z', j')$ if $z \lex\preceq z'$ or $z = z'$ and $j \leq j'$\;
		$M' \gets {}$graded $m \times 0$-matrix with $\rg^{M'} = \rg^M$\;
		$V' \gets {}$empty matrix\;
		$K' \gets {}$graded $0 \times 0$-matrix with $\rg^{K'} = \cg^{M'}$\;
		$m \gets 0 \in \N^n$ \tcp*[r]{index $M_j$ in $M'$}
		\While{$Q \neq \emptyset$}{
			$(z, j) \gets \PopHeap{Q}$\;
			\Forever{}{
				$i \gets \piv M_j$\;
				\lIf{$i = 0$}{append $V'_{m_j}$ to $K'$ and $z$ to $\cg^{K'}$; \Break}
				\lElseIf{$p_{i} = 0$}{$p_i \gets j$; \Break}
				\lElseIf{$\cg^M_{p_i} \not\leq z$}{\PushHeap{Q, (\cg^M_{p_i} \vee z)}; $p_i \gets j$; \Break}
				\Else{
					$M_j \gets M_j + M_{p_i}$\;
					\lIf{$z \neq \cg^M_j$}{$V'_{m_j} \gets V'_{m_j} + V'_{m_{p_i}}$}
				}
			}
			\If(\tcp*[f]{\smash{\parbox[t]{5cm}{$M_j$ cannot be reduced to zero at $\cg^{M}_j$ and therefore is an element of the min. gen. system.}}}){$M_j \neq 0$ and $z = \cg^M_j$}{
				$m_j \gets \text{\# columns in $M'$} + 1$\;
				$M'_{m_j} \gets M_j$ and $\cg^{M'}_{m_j} \gets z$\;
				$V' \gets \Mtx[\big]{V' & 0 \\ 0 & 1}$\;
				$K' \gets \Mtx[\big]{K' \\ 0}$ and $\rg^{K'}_{m_j} \gets z$\;
			}
		}
		\Return{$M', K'$}\;
	}
\end{algorithm}
\end{minipage}

\noindent
\begin{minipage}{\linewidth}
\begin{algorithm}[H]
	\KwIn{A graded $m \times n$-matrix $D$ representing a map of free modules \cite{FugacciKerber:2019a}.}
	\KwOut{A triple $(D', r, c)$.}
	\caption{Splits off direct summands from $D$ that are isomorphic to a ball. The lists $r$ and $c$ are the row and column indices of $D$ that representing the summands $C'$ and $F$ in \eqref{eq:minimize}.}
	\label{algo:minimize}
	\Fn{\Minimize{D'}}{
		$p \gets 0 \in \mathbf{Z}^m$\;
		\For{$j = 1,\dotsc,n$}{
			$i \gets \piv (D_j)$\;
			\lIf{$\rg^D_i \neq \cg^D_j$}{$c \gets c \cup \{j\}$; \Break}
			\lElseIf{$p_i = 0$}{$p_i \gets j$; break}
			\lElse{$D_j \gets D_j + D_{p_i}$}
		}
		\For(\tcp*[f]{embarrassingly parallel}){$j \in c$}{
			\While{$S \coloneqq \Set{i ; D_{ij} \neq 0 \wedge p_i \neq 0}$ is non-empty}{
				$D_j \gets D_j + D_{p_{\max S}}$
			}
		}
		$r \gets \Set{i \leq m ; p_i = 0}$\;
		$D' \gets (D_{ij})_{i \in r, j \in c}$\;
		\Return{$(D', r, c)$}\;
	}
\end{algorithm}
\end{minipage}

\noindent
\begin{minipage}{\linewidth}
\begin{algorithm}[H]
	\caption{Chunk algorithm for chain complexes; see \cref{rmk:minimization}.}
	\label{algo:chain-chunk}
	\KwIn{Graded matrices $[\partial_1], [\partial_2], \dotsc$ forming a chain complex $C_\bullet$ of free $\ZZ$-modules.}
	\KwOut{Graded matrices $[\partial'_1], [\partial'_2], \dotsc$ forming a minimal chain complex $C'_\bullet \simeq C_\bullet$.}
	\Fn{\MinimizeChain{D_\bullet}}{
		$[\partial'_1], r, c \gets \Minimize{[\partial_1]}$\;
		\For{$d = 2, \dotsc$}{
			\settowidth{\dimen0}{$[\partial'_d], r, c$}
			$\mathmakebox[\dimen0][l]{[\partial'_d]} \gets ([\partial_d]_{ij})_{i \in r}$\;
			$\mathmakebox[\dimen0][l]{[\partial'_d], r, c} \gets \Minimize{[\partial'_d]}$\;
			$\mathmakebox[\dimen0][l]{[\partial'_{d-1}]} \gets ([[\partial'_{d-1}])_{ij})_{j \in r}$\;
		}
		\Return $[\partial'_\bullet]$\;
	}
\end{algorithm}
\end{minipage}

\noindent
\begin{minipage}{\linewidth}
\begin{algorithm}[H]
	\caption{Chunk algorithm for cochain complexes; see \cref{rmk:minimization}.}
	\KwIn{Graded matrices $[\d^1], [\d^2], \dotsc$ forming a cochain complex $C^\bullet$ of free $\ZZ$-modules.}
	\KwOut{Graded matrices $[\d'^1], [\d'^2], \dotsc$ forming a minimal cochain complex $C'^\bullet \simeq C^\bullet$.}
	\label{algo:cochain-chunk}
	\Fn{\MinimizeCochain{[\d^\bullet]}}{
		$[\d'^1], r, c \gets \Minimize{[\d^1]}$\;
		\For{$d = 2, \dotsc$}{
			\settowidth{\dimen0}{$[\d'^d], r, c$}
			$\mathmakebox[\dimen0][l]{[\d'^d]} \gets ([\d^d]_{ij})_{j \in c}$\;
			$\mathmakebox[\dimen0][l]{[\d'^d], r, c} \gets \Minimize{[\d'^d]}$\;
			$\mathmakebox[\dimen0][l]{[\d'^{d-1}]} \gets ([\d'^{d-1}])_{ij})_{i \in c}$\;
		}
		\Return $[\d'^\bullet]$\;
	}
\end{algorithm}
\end{minipage}

\noindent
\begin{minipage}{\linewidth}
\begin{algorithm}[H]
	\caption{Factorize a graded matrix through another.}
	\label{algo:factorize}
	\KwIn{A $l \times m$-matrix $L$ in echelon form, and a $l \times n$-matrix $M$ such that $L = MN$ for some $N$.}
	\KwOut{The matrix $N$.}
	\Fn{\Factorize{L, M}}{
		$\mathmakebox[\widthof{$N$}][l]{p} \gets 0 \in \N^l$\;
		$N \gets 0 \in k^{m \times n}$\;
		\lFor{$j = 1,\dotsc,m$}{%
			$p_{\piv M_j} \gets j$%
		}
		\For(\tcp*[f]{embarrassingly parallel}){$j = 1,\dotsc,n$}{
			\While{$L_j \neq 0$}{
				\settowidth{\dimen0}{$N_j$}
				$\mathmakebox[\dimen0][l]{i} \gets \piv L_j$ \tcp*[r]{if $N$ exists, then $p_i \neq 0$}
				$\mathmakebox[\dimen0][l]{L_j} \gets L_j + M_{p_i}$\;
				$\mathmakebox[\dimen0][l]{N_j} \gets N_j + e_{p_i}$\;
			}
		}
		\Return{$N$}
	}
\end{algorithm}
\end{minipage}

\noindent
\begin{minipage}{\linewidth}
\begin{algorithm}[H]
	\caption{Sparsification.}
	\label{algo:sparsification}
	\KwIn{A graded $m \times n$-matrix $M$ such that $\rg^M_i \not\geq \rg^M_j$ if $i \leq j$.}
	\KwOut{A graded matrix $M' = V M$ for some valid invertible upper triangular matrix $V$.}
	$d_j \gets \emptyset$ for all $j$\;
	\For{$i = m,\dotsc,1$}{
		$u \gets \emptyset$\;
		\For{$j = 1,\dotsc,n$ with $M_{ij} \neq 0$}{
			\lIf{$\exists h \in d_j \colon \rg^M_i \leq \rg^{M}_h$}{
				$M_{ij} \gets 0$
			}
			\Else{
				$M'_{ij} \gets M_{ij}$\;
				$u \gets u \cup \{j\}$
			}
		}
		\lIf{$u = \{j\}$ for some $j$}{
			$d_j \gets d_j \cup \{i\}$
		}
	}
\end{algorithm}
\end{minipage}
\fi


\end{document}